\documentclass[11pt,a4paper]{article}
\usepackage[utf8]{inputenc}
\usepackage[left=1in, right=1in, top=1in, bottom=1in]{geometry}
\usepackage{amsmath}
\usepackage{amsfonts}
\usepackage{amssymb}
\usepackage{amsthm}
\usepackage{mathrsfs}
\usepackage{graphicx}
\usepackage{hhline}
\usepackage{color,xcolor}
\usepackage{caption}
\usepackage{enumerate}
\usepackage{mathtools}
\usepackage{hyperref}
\hypersetup{colorlinks = true, linkbordercolor = {white}}
\usepackage{textcomp}
\usepackage{aliascnt}

\usepackage{tikz}
\usetikzlibrary{calc,decorations.pathreplacing,shapes,matrix,arrows,decorations.markings}
\tikzset{draw half paths/.style 2 args={%
  decoration={show path construction,
    lineto code={
      \draw [#1] (\tikzinputsegmentfirst) -- 
         ($(\tikzinputsegmentfirst)!0.5!(\tikzinputsegmentlast)$);
      \draw [#2] ($(\tikzinputsegmentfirst)!0.5!(\tikzinputsegmentlast)$)
        -- (\tikzinputsegmentlast);
    }
  }, decorate
}}
\tikzset{draw third paths/.style 2 args={%
  decoration={show path construction,
    lineto code={
      \draw [#1] (\tikzinputsegmentfirst) -- 
         ($(\tikzinputsegmentfirst)!0.33!(\tikzinputsegmentlast)$);
      \draw [#2] ($(\tikzinputsegmentfirst)!0.67!(\tikzinputsegmentlast)$)
        -- (\tikzinputsegmentlast);
    }
  }, decorate
}}
\tikzset{draw 2third paths/.style 2 args={%
  decoration={show path construction,
    lineto code={
      \draw [#1] (\tikzinputsegmentfirst) -- 
         ($(\tikzinputsegmentfirst)!0.67!(\tikzinputsegmentlast)$);
      \draw [#2] ($(\tikzinputsegmentfirst)!0.33!(\tikzinputsegmentlast)$)
        -- (\tikzinputsegmentlast);
    }
  }, decorate
}}

\tikzset{middlearrow/.style={
        decoration={markings,
            mark= at position 0.5 with {\arrow{#1}} ,
        },
        postaction={decorate}
    }
}

\tikzset{->-/.style={decoration={
  markings,
  mark=at position #1 with {\arrow{>}}},postaction={decorate}}}
  
\newtheorem{theorem}{Theorem}[section]
\newaliascnt{lemma}{theorem}
\newtheorem{lemma}[lemma]{Lemma} 
\aliascntresetthe{lemma}
\newaliascnt{proposition}{theorem}
\newtheorem{proposition}[proposition]{Proposition}   
\aliascntresetthe{proposition}
\newaliascnt{corollary}{theorem}
\newtheorem{corollary}[corollary]{Corollary}
\aliascntresetthe{corollary}
\newaliascnt{question}{theorem}
\newtheorem{question}[question]{Question}
\aliascntresetthe{question}
\newaliascnt{conjecture}{theorem}

\aliascntresetthe{conjecture}

\theoremstyle{definition}

\newaliascnt{example}{theorem}
\newtheorem{example}[example]{Example}
\aliascntresetthe{example}

\newaliascnt{definition}{theorem}
\newtheorem{definition}[definition]{Definition}
\aliascntresetthe{definition}

\theoremstyle{remark}
\newtheorem*{remark}{Remark}

\newcommand{\mc}[1]{\mathcal{#1}}
\newcommand{\mb}[1]{\mathbb{#1}}
\newcommand{\mf}[1]{\mathfrak{#1}}
\newcommand{\ms}[1]{\mathscr{#1}}

\newcommand{\rarr}[1]{\xrightarrow{#1}}

\newcommand{\ri}[1]{R_i^{#1}}

\newcommand{\G}{\Gamma}
\newcommand{\tr}{(\triangle)}

\newcommand{\bart}[1]{\overline{\text{#1}}}

\newcommand{\im}[1]{\text{im }#1}

\newcommand{\dne}{\hfill \qed \vspace{0.3cm}} 
\newcommand{\lra}{(\leftrightarrows)}

\newcommand{\stc}{\; :\; }
\newcommand{\un}[1]{\underline{#1}}

\title{Two Fra\"{i}ss\'{e}-style theorems for homomorphism-homogeneous relational structures}
\author{ Thomas D. H. Coleman\footnotemark[1]}

\begin{document}

\maketitle
\footnotetext[1]{School of Mathematics and Statistics, University of St Andrews, St Andrews, KY16 9SS, United Kingdom. Email: \texttt{tdhc@st-andrews.ac.uk}. This work forms part of the author's PhD thesis at the University of East Anglia.}

\begin{abstract}
In this paper, we state and prove two Fra\"{i}ss\'{e}-style results that cover existence and uniqueness properties for twelve of the eighteen different notions of homomorphism-homogeneity as introduced by Lockett and Truss, and provide forward directions and implications for the remaining six cases. Following these results, we completely determine the extent to which the countable homogeneous undirected graphs (as classified by Lachlan and Woodrow) are homomorphism-homogeneous; we also provide some insight into the directed graph case.

\end{abstract}

\section{Introduction}

A relational first-order structure $\mc{M}$ is \emph{homogeneous} (or \emph{ultrahomogeneous}, depending on source) if every isomorphism between finite substructures of $\mc{M}$ extends to an automorphism of $\mc{M}$. Examples of homogeneous structures include the countable dense linear order without endpoints $(\mb{Q},<)$, the random graph $R$ and the universal tournament $\mb{T}$ \cite{macpherson2011survey}. The celebrated theorem of Fra\"{i}ss\'{e} \cite{fraisse1953certaines} states that a structure $\mc{M}$ is homogeneous if and only if the class $\ms{C}$ of finite substructures of $\mc{M}$ (often called the \emph{age} of $\mc{M}$) satisfies four distinct conditions; furthermore, $\mc{M}$ is unique up to isomorphism if this is the case. A large body of literature, in a range of subjects across mathematics, is devoted to the study of homogeneous structures. Particular areas of interest include classification results \cite{schmerl1979countable} \cite{lachlan1980countable} \cite{cherlin1998classification}, combinatorial aspects \cite{oligomorphic1990} and model-theoretic properties, both as objects in their own right and via a link to $\aleph_0$-categorical structures \cite{hodges1993model}. Through the link to $\aleph_0$-categorical structures and the famous Ryll-Nardzewski theorem (also Engeler, Svenonius \cite{hodges1993model}) there is a well-established connection between automorphisms of countable homogeneous structures and interesting infinite permutation groups \cite{infpermgroups1998}. Many properties of automorphism groups of homogeneous structures have been studied, such as topological dynamics \cite{kechris2005fraisse} and applications to constraint satisfaction problems on infinite domains \cite{bodirsky2012complexity}. Furthermore, the groups themselves have lent themselves to research into properties such as: simplicity \cite{macpherson2011simplicity}; generation \cite{uncountablecofinality2005}; reconstruction of the structure from the automorphism group \cite{barbina2007reconstruction}; the existence of generic elements \cite{truss1992generic}; and the small index property \cite{hodges1993small} \cite{kechris2007turbulence}. An excellent overview of this vast subject is \cite{macpherson2011survey}. 

The idea of homogeneity has been extended to other types of finite partial maps between relational structures. This is the concept of \emph{homomorphism-homogeneity}; originally developed by Cameron and Ne\v{s}et\v{r}il in 2006 \cite{cameron2006homomorphism} and Ma\v{s}ulovi\'{c} in 2007 \cite{masulovic2007homomorphism}. This definition was taken to its logical conclusion in the papers of Lockett and Truss \cite{lockettgeneric} \cite{lockett2014some}, in which they detailed eighteen different notions of homomorphism-homogeneity based on both the finite partial map and the type of endomorphism (see \autoref{defxyhomo}). An example of this is MB-homogeneity; a structure $\mc{M}$ is MB-homogeneous if every monomorphism between finite substructures of $\mc{M}$ extends to a bijective endomorphism of $\mc{M}$. The development of this subject has been rapid, culminating in detailed accounts of homomorphism-homogeneous graphs \cite{rusinov2010homomorphism} \cite{coleman2018permutation}, finite tournaments with loops \cite{ilic2008finite}, and posets \cite{lockett2014some}. Aside from these combinatorial studies, finding analogues of results about automorphism groups for endomorphism monoids is a motivating factor in this subject; examples of these include the development of \emph{oligomorphic transformation monoids} \cite{masulovic2011oligomorphic} \cite{coleman2018permutation} and the idea of \emph{generic endomorphisms} \cite{lockettgeneric}. 


Throughout this article, let $\sigma$ be a countable relational signature. Suppose that $\ms{C}$ is a class of finite $\sigma$-structures. The proof of one direction of 
Fra\"{i}ss\'{e}'s theorem uses 
conditions on $\ms{C}$ to inductively construct a homogeneous structure
$\mc{M}$ whose age is $\ms{C}$. One of these conditions is 
the \emph{joint embedding property} (JEP); this property (along with two others) ensures that we can 
construct a countable $\sigma$-structure $\mc{M}$ with age $\ms{C}$. The second 
is the \emph{amalgamation property} (AP); this ensures that
$\mc{M}$ is homogeneous. This final claim is verified by showing that 
$\mc{M}$ has the \emph{extension property}, a necessary and sufficient condition 
for a countable $\sigma$-structure $\mc{M}$ to be homogeneous.
Fra\"{i}ss\'{e}'s theorem also states that any two homogeneous
$\sigma$-structures with the same age are isomorphic; this is shown using a 
back-and-forth argument between two similarly constructed structures that builds the desired isomorphism. The forth part of 
the argument ensures that the extended map is totally defined; the back part 
ensures that the map is eventually surjective.

Cameron and Ne\v{s}et\v{r}il \cite{cameron2006homomorphism} proved an 
analogue of Fra\"{i}ss\'{e}'s theorem for \emph{MM-homogeneity}, where every monomorphism between finite substructures of 
some structure $\mc{M}$ extends to a monomorphism of $\mc{M}$ (see \autoref{defxyhomo}). 
This proof necessitated modification of the amalgamation property to ensure 
MM-homogeneity; resulting in the \emph{mono-amalgamation property} (MAP). In a 
slight departure to the technique used to prove Fra\"{i}ss\'{e}'s theorem, the 
proof of the analogous theorem for MM-homogeneity in 
\cite{cameron2006homomorphism} utilised a forth alone argument; this is because 
the extended map need not be surjective. On the uniqueness side,
the same article also showed that two MM-structures with the same age may be 
non-isomorphic; instead detailing that two MM-homogeneous structures were unique up to a 
weaker notion called \emph{mono-equivalence}. The proof of this again used a forth alone argument. Finally, work has been done on the case of \emph{HH-homogeneity}, where every finite partial homomorphisms of $\mc{M}$ extends to an endomorphism of $\mc{M}$. The notion of a \emph{homo-amalgamation property} (HAP) first appeared in the preprint of Pech and Pech \cite{pech2011constraint}, in which they used this to prove a version of Fra\"{i}ss\'{e}'s theorem for HH-homogeneity by using a forth alone argument. In addition, they also showed that two HH-homogeneous structures with the same age are homomorphism-equivalent. These results were later published in their following paper \cite{pech2016towards}. Further insights were made by Dolinka \cite{dolinka2014bergman}, who used the HAP (and the equivalent one-point homomorphism extension property (1PHEP)) to determine which structures were both homogeneous and HH-homogeneous.

In the case of MB-homogeneity, a forth alone approach does not suffice. As the 
extended map must be surjective, we are required to use a back-and-forth 
argument. The fact that monomorphisms are not invertible in general necessitates 
the use of a second amalgamation property alongside the MAP of 
\cite{cameron2006homomorphism}; this was defined by Coleman, Evans and Gray 
\cite{coleman2018permutation} using \emph{antimonomorphisms} in the 
\emph{bi-amalgamation property} (BAP). In a similar situation to 
\cite{cameron2006homomorphism}, two MB-homogeneous structures with the same age 
may not be isomorphic but instead are unique up to \emph{bi-equivalence}; the 
proof of this also requires a back-and-forth argument. 

In light of these previous generalisations of Fra\"{i}ss\'{e}'s theorem and the 
multitude of types of homomorphism-homogeneity (see 
\autoref{defxyhomo}), the natural aim would be to find an ``umbrella" version of 
Fra\"{i}ss\'{e}'s theorem; one that encapsulates all possible notions of 
homomorphism-homogeneity. This result would supply Fra\"{i}ss\'{e}'s theorem, 
and the versions of \cite{cameron2006homomorphism} and 
\cite{coleman2018permutation}, as corollaries. Such a theorem could help to determine the extent to which a structure is homomorphism-homogeneous based on structural properties. In turn, this will provide a rich source of oligomorphic 
transformation monoids \cite[Theorem 1.7]{coleman2018permutation}.

However, a compromise must be reached between idealism and practicality for two 
reasons. First, as discussed above, differing approaches are required if the 
extended map is surjective; see the contrast between analogues of Fra\"{i}ss\'{e}'s theorem for MM-homogeneity \cite{cameron2006homomorphism} and 
MB-homogeneity \cite{coleman2018permutation} for a case in point. In the forth alone case, we can utilise a 
single modified amalgamation property in order to construct the structure and 
extend the map. The issue is that monomorphisms and homomorphisms are not ``invertible" in 
general. This is particularly problematic in the homomorphism case; what could you use to accurately describe the `back' condition for homomorphisms, given that the underlying function may not be invertible? As evidenced in \cite{coleman2018permutation}, we need \emph{two} modified amalgamation 
properties in the back-and-forth case; one for the forth part to ensure the extended map is totally defined, and one for the back part to ensure the resulting map is surjective. Second, some kinds 
of homomorphism-homogeneity are easier to deal with than others. There is a 
distinct dichotomy in the set of notions of homomorphism-homogeneity, split between those whose extended maps are not 
necessarily the same ``type" as the partial map (such as MH-homogeneity, in that 
a homomorphism is not necessarily an monomorphism), and those whose extended 
maps are definitely of the same type than the partial map (such as MM, or 
MI-homogeneity). The former case causes issues in inductively constructing a 
structure due to the lack of certainty about the extended map; this is discussed 
in further detail in \autoref{sforth}.

The first of these reasons therefore necessitate \emph{two} similar but markedly 
different theorems (\autoref{Forth} and \autoref{BackForth}) based on whether or 
not the proof uses forth alone or a back-and-forth argument; these form the central theorems of the paper. The second only allows 
the two theorems to cover twelve of the eighteen different notions of 
homomorphism-homogeneity. In the statement of the theorems below, what constitutes the ``relevant" amalgamation property 
and notion of equivalence will be explained in Sections \ref{sforth} and \ref{sbackforth}.

\begin{theorem}\label{Forth} Let XY $\in \{$II, MI, MM, HI, HM, HH$\}$.
\begin{enumerate}[(1)]
\item If $\mc{M}$ is an XY-homogeneous $\sigma$-structure, then Age$(\mc{M})$ 
has the relevant amalgamation property.
\item If $\ms{C}$ is a class of finite $\sigma$-structures with countably many 
isomorphism types, is closed under isomorphisms and substructures, has the JEP 
and the relevant amalgamation property, then there exists a XY-homogeneous 
$\sigma$-structure $\mc{M}$ with age $\ms{C}$.
\item Any two XY-homogeneous $\sigma$-structures with the same age are 
equivalent up to a relevant notion of equivalence.
\end{enumerate}
\end{theorem}

\begin{theorem}\label{BackForth} Let XZ $\in \{$IA, MA, MB, HA, HB, HE$\}$.
\begin{enumerate}[(1)]
\item If $\mc{M}$ is an XZ-homogeneous $\sigma$-structure, then Age$(\mc{M})$ 
has the two relevant amalgamation properties.
\item If $\ms{C}$ is a class of finite $\sigma$-structures with countably many 
isomorphism types, is closed under isomorphisms and substructures, has the JEP 
and the two relevant amalgamation properties, then there exists a XZ-homogeneous 
$\sigma$-structure $\mc{M}$ with age $\ms{C}$.
\item Any two XZ-homogeneous $\sigma$-structures with the same age are 
equivalent up to a relevant notion of equivalence.
\end{enumerate}
\end{theorem}

While not the ideal ``umbrella'' theorem, these two results are still useful in 
determining the extent to which a structure is homomorphism-homogeneous; thus 
providing interesting examples of oligomorphic transformation monoids. To that end, this article is dedicated to the proof of these two theorems; as 
well as determining a complete picture of homomorphism-homogeneity for some 
well-known structures. 

\autoref{sforth} begins by defining the eighteen different notions of homomorphism-homogeneity, and contains the proof of \autoref{Forth} split into three propositions (\ref{xyap}, \ref{xyexist}, \ref{xyunique}) that correspond to the three points of the theorem. \autoref{smultanti} utilises the idea of the converse of a function to introduce the concept of an \emph{antihomomorphism} between two $\sigma$-structures; essentially, this is the preimage of a homomorphism that preserves non-relations. This machinery underpins the `back' condition that is used to extend a map between finite substructures of a $\sigma$-structure $\mc{M}$ to a surjective endomorphism of $\mc{M}$. Following this, \autoref{sbackforth} is dedicated to the proof of \autoref{BackForth}. In \autoref{smhhc} we introduce the idea of a \emph{maximal homomorphism-homogeneity class} (mhh-class), and determine mhh-classes for every countable homogeneous undirected graph in the classification of Lachlan and Woodrow \cite{lachlan1980countable}, as well as turning our attention to the directed graph case.

Throughout the article, we will write $\un{x}$ to mean an $n$-tuple of some set; this non-standard notation is motivated by the use of barred notation to mean converses of functions, which appear more regularly. The notation $^{-1}$ is reserved exclusively for the inverse of a function. For some countable indexing set $I$, we define $\sigma = \{R_i\stc i\in I\}$ to be a relational signature. Usually, $\mc{M}$ will denote a countable $\sigma$-structure on domain $M$. The \emph{age} Age$(\mc{M})$ of $\mc{M}$ is the class of all finite structures that can be embedded in $\mc{M}$. For more on the introductory concepts of model theory, \cite{hodges1993model} is a good place to start.

\section{Proof of \autoref{Forth}}\label{sforth}

We recall the eighteen different notions of homomorphism-homogeneity as developed in the two papers of Lockett and Truss \cite{lockettgeneric} \cite{lockett2014some}. Following their lead, we denote each type of endomorphism by a symbol: H for endomorphism, E for epimorphism, M for monomorphism, B for bimorphism, I for embedding and A for automorphism. We cannot assert that a finite partial map is surjective; there is no well defined notion of a finite partial epimorphism, for instance. Therefore, there are only three types of finite partial map of a structure: H for homomorphism, M for monomorphism, and I for embedding. Without loss of generality, maps between finite substructures can be taken to be surjective. 

\begin{definition}\label{defxyhomo} Let $\mc{M}$ be a first-order structure, and take X $\in \{\text{H},\text{M},\text{I}\}$ and Y $\in \{\text{H},\text{E},\text{M},\text{B},\text{I},\text{A}\}$. Say that $\mc{M}$ is \emph{XY-homogeneous} if every finite partial map of type X of $\mc{M}$ extends to a map of type Y of $\mc{M}$. We denote the collection of all notions of homomorphism-homogeneity by $\mf{H}$.
Furthermore, we denote the \emph{class of all XY-homogeneous structures} by XY, and say that $\mb{H}$ is the set of all classes of XY-homogeneous structures.
\end{definition}

For example, a structure $\mc{M}$ is HE-homogeneous if every finite partial homomorphism (H) of $\mc{M}$ extends to a epimorphism (E) of $\mc{M}$. Regular homogeneity (as in \cite{macpherson2011survey}, for instance) corresponds to IA-homogeneity using this notation. All possible types of homomorphism-homogeneity given in \autoref{defxyhomo} are outlined in \autoref{xyhomo}. 

\begin{table}[h]
\renewcommand{\arraystretch}{1.25}
\centering
\begin{tabular}{c c c c}
\hline 
 & isomorphism (I) & monomorphism (M) & homomorphism (H) \\ 
\hline \hline
End$(\mc{M})$ (H) & IH & MH & HH \\ 
\hline 
Epi$(\mc{M})$ (E) & IE & ME & HE \\
\hline 
Mon$(\mc{M})$ (M) & IM & MM & HM \\ 
\hline 
Bi$(\mc{M})$ (B) & IB & MB & HB \\ 
\hline 
Emb$(\mc{M})$ (I) & II & MI & HI \\ 
\hline 
Aut$(\mc{M})$ (A) & IA & MA & HA \\ 
\hline 
\end{tabular}
\caption{Table of XY-homogeneity: $\mc{M}$ is XY-homogeneous if a finite partial map of type X (column) extends to a map of type Y (row) in the associated monoid. The collection of all notions of homomorphism-homogeneity is denoted by $\mf{H}$.}\label{xyhomo}
\end{table}

It is important to make the distinction between a notion of homomorphism-homogeneity and the associated class of homomorphism-homogeneous structures. For example, II-homogeneity and IA-homogeneity represent two different notions of homomorphism-homogeneity under consideration. As outlined in \autoref{xyhomo}, II-homogeneity is where every finite partial isomorphism extends to an \emph{embedding}; IA-homogeneity is where every finite partial isomorphism extends to an automorphism. However, the classes II and IA of homomorphism-homogeneous structures coincide. For countable structures, it was shown by Lockett and Truss \cite{lockett2014some} that a structure is II (MI, HI)-homogeneous if and only if it is IA (MA, HA)-homogeneous; that is, II = IA, MI = MA and HI = HA. This difference between notions of homomorphism-homogeneity and classes of homomorphism-homogeneous structures is apparent in \autoref{sbackforth}, where we re-prove this result of \cite{lockett2014some} from a Fra\"{i}ss\'{e}-theoretic perspective.

It follows that some notions of homomorphism-homogeneity are stronger than others. For instance, as every bimorphism is a monomorphism, it follows that every MB-homogeneous structure is also MM-homogeneous. Similarly, as an isomorphism is both a monomorphism and a homomorphism, it follows that if a structure $\mc{M}$ is XY-homogeneous then it is also IY-homogeneous. \label{invcor} This natural concept inversely corresponds to a natural containment order on the set $\mb{H}$ of homomorphism-homogeneity classes; see \autoref{hhdiag} for a diagram of this order. Notice that the stronger the notion of homomorphism-homogeneity, the class of countable structures that satisfy that notion is smaller. This difference is explained in more detail in \autoref{smhhc}.

\begin{center}
\begin{figure}[h]
\centering
\begin{tikzpicture}[node distance=2cm,inner sep=0.8mm,scale=1.2]
\node (0) at (0,0) {HA};
\node (1) at (-1,1) {MA};
\node (2) at (0,1) {HI};
\node (3) at (1,1) {HB};
\node (4) at (-2,2) {IA};
\node (5) at (-1,2) {MI};
\node (6) at (0,2) {MB};
\node (7) at (1,2) {HM};
\node (8) at (2,2) {HE};
\node (9) at (-2,3) {II};
\node (10) at (-1,3) {IB};
\node (11) at (0,3) {MM};
\node (12) at (1,3) {ME};
\node (13) at (2,3) {HH};
\node (14) at (-1,4) {IM};
\node (15) at (0,4) {IE};
\node (16) at (1,4) {MH};
\node (17) at (0,5) {IH};
\draw (0) -- (1);
\draw [thick,double] (0) -- (2);
\draw (0) -- (3);
\draw (1) -- (4);
\draw [thick,double] (1) -- (5);
\draw (1) -- (6);
\draw (2) -- (5);
\draw (2) -- (7);
\draw (3) -- (6);
\draw (3) -- (7);
\draw (3) -- (8);
\draw [thick,double] (4) -- (9);
\draw (4) -- (10);
\draw (5) -- (9);
\draw (5) -- (11);
\draw (6) -- (10);
\draw (6) -- (11);
\draw (6) -- (12);
\draw (7) -- (11);
\draw (7) -- (13);
\draw (8) -- (12);
\draw (8) -- (13);
\draw (9) -- (14);
\draw (10) -- (14);
\draw (10) -- (15);
\draw (11) -- (14);
\draw (11) -- (16);
\draw (12) -- (15);
\draw (12) -- (16);
\draw (13) -- (16);
\draw (14) -- (17);
\draw (15) -- (17);
\draw (16) -- (17);
\end{tikzpicture}
\caption{The set $\mb{H}$ of homomorphism-homogeneity classes for countable first-order structures, partially ordered by inclusion. Lines indicate inclusion, double lines indicate equality.}\label{hhdiag}
\end{figure}
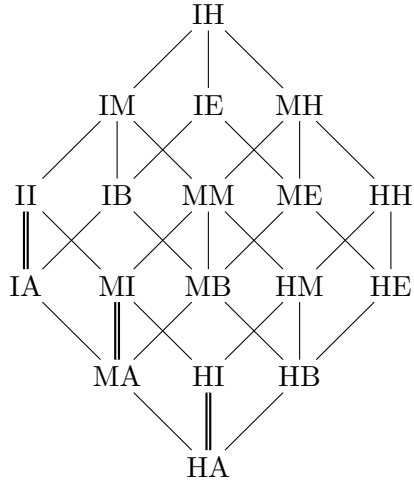
\end{center}
As discussed in the introduction, it is necessary to partition $\mf{H}$ into two pieces based on whether 
or not the extended map is surjective. This represents the division between 
cases where a forth alone argument will suffice and the other when we require a 
back-and-forth construction. Furthermore, there are some elements of $\mf{H}$ 
that are weaker notions of homogeneity than others. These are of the form XY 
where a map of type Y does not necessarily imply that it is a map of type X; for 
instance, a homomorphism is not necessarily a monomorphism. These phenomena 
motivate the division of $\mf{H}$ into the following subsets:
\begin{itemize}
\item forth alone $\mf{F} = \left\{\text{XY }\in\mf{H}\stc  \text{X, Y}\in\{\text{H, M, I}\}\right\}$;
\item back-and-forth $\mf{B} = \{$XZ $\in\mf{H}$ : X $\in\{$H, M, I$\}$, Z 
$\in\{$E, B, A$\}\}$;
\item no implication $\mf{N} = \{$IH, IE, IM, IB, MH, ME$\}$;
\item implication $\mf{I} = \mf{H}\smallsetminus\mf{N}$.\label{fbni}
\end{itemize}

This partitions $\mf{H}$ into four parts based on the intersections of 
$\mf{B},\mf{F}$ with $\mf{N},\mf{I}$ (see \autoref{int}, where the boxes 
represent intersections).
\begin{center}
\begin{figure}[h]
\centering
\begin{tikzpicture}[node distance=2cm,inner sep=0.8mm,scale=0.8]
\node (0) at (3,-1) {HA};
\node (1) at (2,-2) {MA};
\node (2) at (-1,-1) {HI};
\node (3) at (2,-1) {HB};
\node (4) at (3,-2) {IA};
\node (5) at (-2,-2) {MI};
\node (6) at (1,-2) {MB};
\node (7) at (-2,-1) {HM};
\node (8) at (1,-1) {HE};
\node (9) at (-1,-2) {II};
\node (10) at (2,1) {IB};
\node (11) at (-3,-2) {MM};
\node (12) at (3,1) {ME};
\node (13) at (-3,-1) {HH};
\node (14) at (-2,1) {IM};
\node (15) at (1,1) {IE};
\node (16) at (-1,1) {MH};
\node (17) at (-3,1) {IH};
\node (F) at (-2,2.5) {$\mf{F}$};
\node (B) at (2,2.5) {$\mf{B}$};
\node (N) at (-4.5,1) {$\mf{N}$};
\node (I) at (-4.5,-1) {$\mf{I}$};
\draw (-4,-3) rectangle (4,2);
\draw (-4,-3) rectangle (0,0);
\draw (0,0) rectangle (4,2);
\end{tikzpicture}
\caption{Diagram illustrating the subsets $\mf{F}$, $\mf{B}$, $\mf{N}$, and $\mf{I}$ of $\mf{H}$.}\label{int}
\end{figure}
\end{center}

We move on to establish the machinery required for the proof of \autoref{Forth}. This result deals with types of 
homomorphism-homogeneity in $\mf{F}$ (see \autoref{int}); those that only 
require a forth construction to prove. Consequently, we have that X,Y $\in\{$H, 
M, I$\}$ throughout this section. When we say a \emph{map of type X}, we are 
referring to this instance; so if $\alpha$ is a map of type H, it is a 
homomorphism. Notice that I $\subseteq$ M $\subseteq$ H.

A critical step in Fra\"{i}ss\'{e}'s proof is the establishment that the inductively constructed structure $\mc{M}$ is homogeneous; nominally by showing that $\mc{M}$ satisfies the \emph{extension property}, a necessary and sufficient condition for homogeneity. Since different kinds of homomorphism-homogeneity rely on extending different kinds of maps, this property needs to be generalised and then shown to be an equivalent condition to the relevant notion of homomorphism-homogeneity. To that end, we define the \emph{XY-extension property} (XYEP), where X,Y $\in\{$H, M, I$\}$:

\begin{quotation}(XYEP) A structure $\mc{M}$ with age $\ms{C}$ has the 
\emph{XYEP} if for all $A\subseteq B\in\ms{C}$ and maps $f:A\rarr{}\mc{M}$ of 
type X, there exists a map $g:B\rarr{}\mc{M}$ of type Y extending $f$.
\end{quotation}

For example, the standard extension property in Fra\"{i}ss\'{e}'s theorem is the IIAP, and the \emph{mono-extension property} of \cite{cameron2006homomorphism} is the MMAP.

Ideally, we would like to say a structure $\mc{M}$ is XY-homogeneous if and only if $\mc{M}$ has the 
XYEP, generalising the observation of Fra\"{i}ss\'{e}. However, complications occur in the proof of the converse direction of this statement; this 
is due to the inductive construction of the extended map. For example, suppose 
that $\mc{M}$ has the IMEP and that $f:A\rarr{}B$ is an isomorphism. Extending 
this using the IMEP gives a monomorphism $g:A'\to B'$ where $A\subseteq A'$ and 
$B\subseteq B'$. However, $g$ is a monomorphism between finite substructures; 
and so in general, we cannot extend $g$ to another monomorphism $h$ between 
finite substructures. The only way we can continue extending is if the map of 
type Y is also of type X. This behaviour is the motivating factor in splitting 
$\mf{H}$ into $\mf{I}$ and $\mf{N}$ (see \pageref{fbni}). In light of this, we show that the XYEP is 
a necessary condition for XY-homogeneity in general, and that it is also 
sufficient when the extended map of type Y is also a map of type X.

\begin{proposition}\label{xyep} Let $\mc{M}$ be a countable $\sigma$-structure 
with age $\ms{C}$. 
\begin{enumerate}[(1)]
\item Suppose that XY $\in\mf{F}$. If $\mc{M}$ is XY-homogeneous, then $\mc{M}$ 
has the XYEP.
\item Suppose that XY $\in\mf{F}\cap\mf{I}$. If $\mc{M}$ has the XYEP, then 
$\mc{M}$ is XY-homogeneous.
\end{enumerate}
\end{proposition}

\begin{proof} (1) Let $A\subseteq B\in\ms{C}$ and $f:A\to\mc{M}$ be a map of 
type X. As Age$(\mc{M}) = \ms{C}$, there exists an isomorphism $\theta:B\to 
B\theta\subseteq\mc{M}$. Therefore, $\theta^{-1}f: B\theta\to Af$ is a map of 
type X between finite substructures of $\mc{M}$. As $\mc{M}$ is XY-homogeneous, 
extend $\theta^{-1}f$ to a map $\alpha:\mc{M}\to\mc{M}$ of type Y. Hence, 
$\theta\alpha:B\to\mc{M}$ is a map of type Y. It remains to show that 
$\theta\alpha$ extends $f$; for any $a\in A$ it follows that \[af = a\theta\theta^{-1}f 
= a\theta\alpha\] as $\alpha$ extends $\theta^{-1}f$. Therefore $\mc{M}$ has the 
XYEP.\\

(2) Suppose that $f:A\rarr{}B$ is a map of type X between finite substructures 
of $\mc{M}$. We use a forth argument to extend $f$ to a map $\alpha$ of type Y. 
As $\mc{M}$ is countable, we enumerate the points of $M = \{m_0,m_1,\ldots\}$. Set 
$A = A_0, B = B_0$ and $f = f_0$ and assume that we have extended $f$ to a map 
$f_k:A_k\rarr{}B_k$, where $A_i\subseteq A_{i+1}$ and $B_i\subseteq B_{i+1}$ for 
all $0\leq i\leq k-1$. At most, we can assume that $f_k$ is a map of type Y. 
Select $m_i\in\mc{M}\smallsetminus A_k$, where $i$ is the least natural number 
such that $m_i\notin A_k$. We can see that $A_k\cup\{m_i\}\subseteq \mc{M}$ 
belongs to $\ms{C}$. As XY $\in\mf{I}$, the map $f_k$ of type Y is also of type X; 
so use the XYEP to find a map $f_{k+1}:A_k\cup\{m_i\}\rarr{}\mc{M}$ of type Y 
extending $f_k$. Repeating this process infinitely many times, ensuring that 
each $m_i$ appears at some stage, extends $f$ to a map 
$\alpha:\mc{M}\rarr{}\mc{M}$ of type Y; so $\mc{M}$ is XY-homogeneous.
\end{proof}

We now move to the proof of \autoref{Forth}. Our eventual aim is to construct a countable structure $\mc{M}$ with age 
$\ms{C}$, where $\mc{M}$ is XY-homogeneous. Recall (from \cite{macpherson2011survey}) that a class $\ms{C}$ of finite $\sigma$-structures has the 
\emph{joint embedding property} (JEP) if for all $A,B\in\ms{C}$ there exists a 
$D\in\ms{C}$ such that $D$ jointly embeds $A$ and $B$. This property, along with $\ms{C}$ being closed under substructures and isomorphisms, and having countably many isomorphism types, is required to construct a countable structure $\mc{M}$ with age $\ms{C}$; it 
has nothing to do with the homogeneity of the structure $\mc{M}$. In Fra\"{i}ss\'{e}'s theorem, it is the 
amalgamation property that is central to ensuring that the constructed structure is homogeneous; following the lead of \cite{cameron2006homomorphism}, this must be generalised in order to ensure XY-homogeneity. So to construct 
a countable, XY-homogeneous structure $\mc{M}$ with age $\ms{C}$, the class of finite structures $\ms{C}$ must have the JEP and 
some generalised amalgamation property.

Since different types of homogeneity require different amalgamation properties, 
it then makes sense to define an ``umbrella'' condition; one that 
encompasses every required amalgamation property. This is the 
\emph{XY-amalgamation property} (XYAP), where X,Y $\in\{$H, M, I$\}$:

\begin{quotation}(XYAP) Let $\ms{C}$ be a class of finite $\sigma$-structures. Then 
$\ms{C}$ has the \emph{XYAP} if for all $A,B_1,B_2\in\ms{C}$, map 
$f_1:A\rarr{}B_1$ of type X and embedding $f_2:A\rarr{}B_2$, there exists a 
$D\in\ms{C}$, embedding $g_1:B_1\rarr{}D$ and map $g_2:B_2\rarr{}D$ of type Y 
such that $f_1g_1 = f_2g_2$ (see \autoref{dxyap}).
\end{quotation}

\begin{figure}[h]
\centering
\begin{tikzpicture}[node distance=2cm]
\node(D) {$\exists D$};
\node(R) [below right of=D] {$B_2$};
\node(L) [below left of=D] {$B_1$};
\node(H) [below right of=L] {$A$};
\path[->,thick,dotted] (R) edge node[above right] {$g_2$} (D);
\path[right hook->,thick,dotted] (L) edge node[above left] {$g_1$} (D);
\path[right hook->] (H) edge node[below right] {$f_2$} (R);
\path[->] (H) edge node[below left] {$f_1$} (L);
\end{tikzpicture}
\caption{Diagram of the XY-amalgamation property (XYAP).}\label{dxyap}
\end{figure}
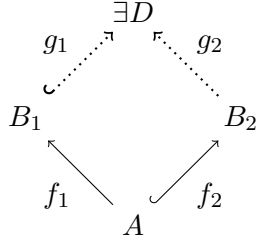

Based on choices for X and Y, the XYAP yields nine different amalgamation 
properties; one for each notion of XY-homogeneity in $\mf{F}$. For instance, the 
IIAP is the standard amalgamation property, the MMAP is the MAP in 
\cite{cameron2006homomorphism} and the HHAP is the HAP from 
\cite{dolinka2014bergman}. These are the relevant amalgamation properties as alluded to in \autoref{Forth}. Our next result demonstrates \autoref{Forth} (1).

\begin{proposition}[\autoref{Forth} (1)]\label{xyap} Suppose that XY $\in\mf{F}$. If a countable
$\sigma$-structure $\mc{M}$ is XY-homogeneous, then Age$(\mc{M})$ has the XYAP.
\end{proposition}

\begin{proof} Suppose that $A,B_1,B_2\in$ Age$(\mc{M})$, $f_1:A\rarr{}B_1$ is a 
map of type X and $f_2:A\rarr{}B_2$ is an embedding. Without loss of generality, 
suppose that $f_2$ is the inclusion map and that $A,B_1,B_2\subseteq \mc{M}$. 
Using XY-homogeneity of $\mc{M}$, extend $f_1:A\rarr{}B$ to a map 
$\alpha:\mc{M}\rarr{}\mc{M}$ of type Y. Set $D = B_1\cup B_2\alpha$ and induce 
the structure on $D$ with relations from $\mc{M}$. Finally, take 
$g_1:B_1\rarr{}D$ to be the inclusion map and define $g_2$ to be the map $g_2 = 
\alpha|_{B_2}:B_2\rarr{}D$ of type Y. We can see that $f_1g_1 = f_2g_2$ and so 
these choices verify the XYAP for Age$(\mc{M})$.
\end{proof}

Now, we proceed with the proof of \autoref{Forth} (2). As in \cite{cameron2006homomorphism}, different stages of the inductive construction are achieved at even and odd steps.

\begin{proposition}[\autoref{Forth} (2)]\label{xyexist} Suppose that XY $\in\mf{F} \cap \mf{I}$. Let 
$\ms{C}$ be a class of finite $\sigma$-structures that is closed under 
isomorphism and substructures, has countably many isomorphism types, and has the 
JEP and XYAP. Then there exists a countable XY-homogeneous $\sigma$-structure $\mc{M}$ 
with age $\ms{C}$. 
\end{proposition}

\begin{proof} 
Along the lines of similar proofs of Fra\"{i}ss\'{e}'s theorem (see \cite{infpermgroups1998}, \cite{cameron2006homomorphism}), the idea is to construct $\mc{M}$ over countably many stages, assuming that $M_k$ has been constructed at some stage $k\in\mb{N}$, with $M_0$ being some $A\in\ms{C}$. As the number of isomorphism types in $\ms{C}$ is countable, we can choose a countable set $S$ of pairs $(A,B)$ with $A\subseteq B\in\ms{C}$ such that every pair $A'\subseteq B'\in\ms{C}$ is represented by some pair $(A,B)\in S$. Define a bijection $\beta:\mb{N}\times \mb{N}\to \mb{N}$ such that $\beta(i,j) \geq i$.

Assume first that $k$ is even. As $\ms{C}$ has countably many isomorphism types, we can enumerate the isomorphism types of $\ms{C}$ by $\{M_0 = T_0,T_2,T_4,\ldots\}$. Use the JEP to find a structure $D$ that contains both $M_k$ and a copy of $C\cong T_k\in\ms{C}$; define $M_{k+1}$ to be this structure $D$. Now suppose that $k$ is odd. Let $L = (A_{kj}, B_{kj}, f_{kj})_{j\in\mb{N}}$ be the list of all triples $(A,B,f)$ such that $(A,B)\in S$ and $f:A\to M_k$ is a map of type X. This list is countable as $S$ is and there are finitely many maps of type X from $A$ into $M_k$. Let $k = \beta(i,j)$. Then as $\beta(i,j)\geq i$, the map $f_{ij}:A_{ij}\to M_i\subseteq M_k$ exists. Therefore, we can use the XYAP to define $M_{k+1}$ such that $M_k\subseteq M_{k+1}$ and the map $f_{ij}:A_{ij}\to M_k$ of type X extends to some map $g_{ij}:B_{ij}\rarr{}M_{k+1}$ of type Y. This ensures that every possible XY-amalgamation occurs.

Define $\mc{M} = \bigcup_{k\in\mb{N}}M_k$. Our construction ensures that every isomorphism type of $\ms{C}$ appears at a 0 mod 3 stage, so every structure in $\ms{C}$ embeds into $\mc{M}$. Conversely, we have that $M_k\in\ms{C}$ for all $k\in\mb{N}$. As $\ms{C}$ is closed under substructures, every structure that embeds into $\mc{M}$ is in $\ms{C}$, showing that Age$(\mc{M}) = \ms{C}$. 

It remains to show that $\mc{M}$ is XY-homogeneous. As XY $\in\mf{F} \cap \mf{I}$, it is enough to show that $\mc{M}$ has the XYEP by \autoref{xyep}. So assume that $A \subseteq B\in\ms{C}$ and that $f:A\to\mc{M}$ is a map of type X. As $Af$ is finite, it follows that there exists $j\in \mathbb{N}$ such that $Af\subseteq M_j$. Furthermore, there exists a triple $(A_{j\ell},B_{j\ell},f_{j\ell})\in L$ such that there exists an isomorphism $\theta:B\to B_{j\ell}$ with $A\theta|_A = A_{j\ell}$ and $f = \theta|_Af_{j\ell}$. Define $n = \beta(j,\ell)$; as $n \geq j$, it follows that $Af = A_{j\ell}f_{j\ell}\subseteq M_j \subseteq M_n$. Here, $M_{n+1}$ is constructed by XY-amalgamating $M_n$ and $B_{j\ell}$ over $A_{j\ell}$; this provides an extension $g_{j\ell}:B_{j\ell} \to M_{n+1}$ of type Y to the map $f_{j\ell}$ of type X. It follows that $g = \theta g_{j\ell}:B\to\mc{M}$ is a map of type Y extending the map $f$ of type $X$ (see \autoref{xyamaldiag} for a diagram). Therefore, $\mc{M}$ has the XYEP, proving that it is XY-homogeneous.\end{proof}
\begin{center}
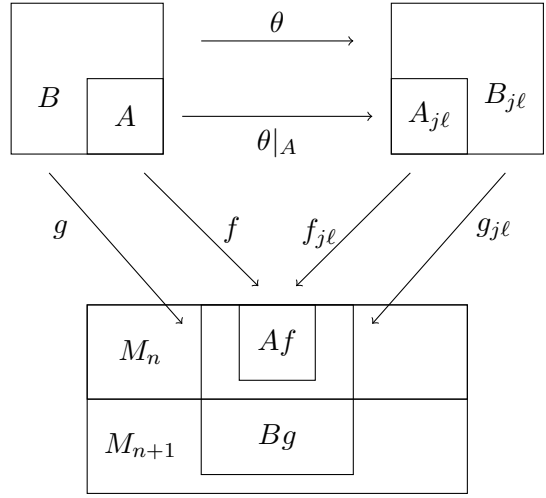
\begin{figure}[h]
\centering
\begin{tikzpicture}[node distance=1.8cm,inner sep=0.7mm,scale=1]
\draw (-3.5,3) rectangle (-1.5,1);
\draw (-2.5,2) rectangle (-1.5,1);
\node (a11) at (-2,1.5) {$A$};
\node (b11) at (-3,1.75) {$B$};

\draw (3.5,3) rectangle (1.5,1);
\draw (2.5,2) rectangle (1.5,1);
\node (a1) at (2,1.5) {$A_{j\ell}$};
\node (b) at (3,1.75) {$B_{j\ell}$};

\draw[->] (-1.25,1.5) -- (1.25,1.5);
\node (e1) at (0,1.15) {$\theta|_A$};

\draw[->] (-1,2.5) -- (1,2.5);
\node (e) at (0,2.75) {$\theta$};

\draw (-2.5,-1) rectangle (2.5,-2.25);
\draw (-2.5,-1) rectangle (2.5,-3.5);
\node at (-1.8,-1.625) {$M_n$};
\node at (-1.8,-2.875) {$M_{n+1}$};
\draw (-1,-3.25) rectangle (1,-1);
\draw (-0.5,-2) rectangle (0.5,-1);
\node (a) at (0,-1.5) {$Af$};
\node (b) at (0,-2.75) {$Bg$};
\draw[<-] (0.25,-0.75) -- (1.75,0.75);
\node (alpha) at (0.55,-0.05) {$f_{j\ell}$};
\draw[<-] (1.25,-1.25) -- (3,0.75);
\node (alpha) at (2.85,0) {$g_{j\ell}$};
\draw[<-] (-0.25,-0.75) -- (-1.75,0.75);
\node (f) at (-0.6,0) {$f$};
\draw[<-] (-1.25,-1.25) -- (-3,0.75);
\node (h1alpha) at (-2.85,0) {$g$};
\draw (-2.5,-2.25) -- (2.5,-2.25);
\end{tikzpicture}
\caption{Diagram of maps in the proof of \autoref{xyexist}. The map $g = \theta g_{j\ell}:B\to\mc{M}$ of type Y extends the map $f$ of type X.}\label{xyamaldiag}
\end{figure}
\end{center} 

All that remains to show is part (3) of \autoref{Forth}. It was previously 
mentioned in \cite{cameron2006homomorphism} that two MM-homogeneous structures 
with the same age need not be isomorphic, but are instead 
\emph{mono-equivalent}. This inspires another collective definition; and is the relevant notion of equivalence as mentioned in the statement of \autoref{Forth}.

\begin{definition} Let $\mc{M,N}$ be $\sigma$-structures and suppose that Y 
$\in\{$H, M, I$\}$. Say that $\mc{M}$ and $\mc{N}$ are \emph{Y-equivalent} if 
Age$(\mc{M}) =$ Age$(\mc{N})$ and every embedding $f:A\to\mc{N}$ from a finite 
substructure $A$ of $\mc{M}$ can be extended to a map $g:\mc{M}\to \mc{N}$ of 
type Y, and vice versa. 
\end{definition}

Note that if two structures $\mc{M,N}$ are M-equivalent, then they are 
mono-equivalent in the sense of \cite{cameron2006homomorphism}. If two 
structures $\mc{M,N}$ are I-equivalent, then they are mutually embeddable. 

\begin{proposition}[\autoref{Forth} (3)]\label{xyunique} Let $\mc{M,N}$ be countable 
$\sigma$-structures, and suppose that XY $\in\mf{F}\cap\mf{I}$.
\begin{enumerate}[(1)]
\item Suppose that $\mc{M,N}$ are Y-equivalent. Then $\mc{M}$ is XY-homogeneous 
if and only if $\mc{N}$ is.
\item If $\mc{M,N}$ are XY-homogeneous and Age$(\mc{M})$ = Age$(\mc{N})$ then 
$\mc{M,N}$ are Y-equivalent.
\end{enumerate}
\end{proposition}

\begin{proof} 
(1) It suffices to show that $\mc{N}$ has the XYEP by \autoref{xyep} (2). 
Suppose then that $A\subseteq B\in$ Age$(\mc{N})$ and there exists a map 
$f:A\rarr{}A'\subseteq\mc{N}$ of type X. Note that $A$ need not be isomorphic to 
$A'$. As Age$(\mc{M})=$ Age$(\mc{N})$, there exists a copy $A''$ of $A'$ in 
$\mc{M}$; fix an embedding $e:A'\to A''$ between the two. Therefore, 
$e^{-1}:A''\to A'$ is a isomorphism from a finite substructure of $\mc{M}$ into 
$\mc{N}$; as the two are Y-equivalent, we extend this to a map 
$\alpha:\mc{M}\to\mc{N}$ of type Y. Now, define a map $h = fe: A\to A''$; this 
is a map of type X from $A$ into $\mc{M}$. Since $\mc{M}$ is XY-homogeneous, it 
has the XYEP by \autoref{xyep} (1) and so we extend $h$ to a map $h':B\to 
\mc{M}$ of type Y. Now, the map $h'\alpha:B\rarr{}\mc{N}$ is a map of type Y; we 
need to show it extends $f$. So using the facts that $\alpha$ extends $e^{-1}$ 
and $h'$ extends $h = fe$, we have that for all $a\in A$: \[af = afee^{-1} = 
ah'\alpha.\] Therefore $\mc{N}$ has the XYEP. A diagram of this process can be found in \autoref{xyequivdiag}.

\begin{center}
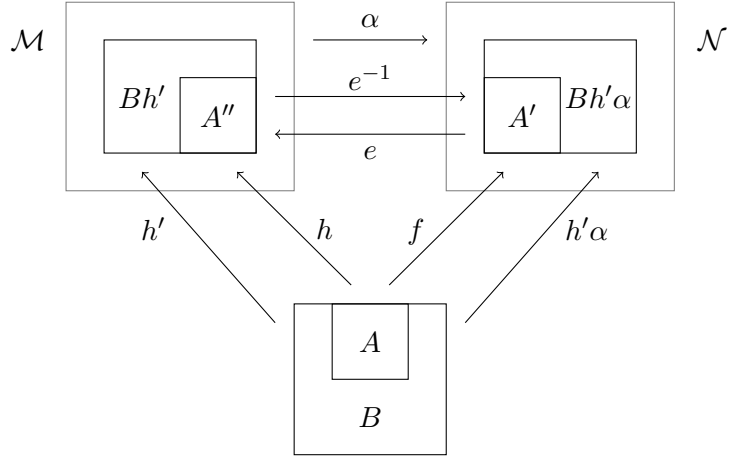
\begin{figure}[h]
\centering
\begin{tikzpicture}[node distance=1.8cm,inner sep=0.7mm,scale=1]

\draw[black!50!white] (-4,0.5) rectangle (-1,3);
\draw[black!50!white] (1,0.5) rectangle (4,3);
\node(m) at (-4.5,2.5) {$\mc{M}$}; 
\node(n) at (4.5,2.5) {$\mc{N}$};

\draw (-3.5,2.5) rectangle (-1.5,1);
\draw (-2.5,2) rectangle (-1.5,1);
\node (a11) at (-2,1.5) {$A''$};
\node (b11) at (-3,1.75) {$Bh'$};

\draw (3.5,2.5) rectangle (1.5,1);
\draw (2.5,2) rectangle (1.5,1);
\node (a1) at (2,1.5) {$A'$};
\node (b) at (3,1.75) {$Bh'\alpha$};

\draw[->] (-0.75,2.5) -- (0.75,2.5);
\node (alpha) at (0,2.75) {$\alpha$};

\draw[->] (-1.25,1.75) -- (1.25,1.75);
\node (e1) at (0,2) {$e^{-1}$};

\draw[<-] (-1.25,1.25) -- (1.25,1.25);
\node (e) at (0,1) {$e$};

\draw (-1,-3) rectangle (1,-1);
\draw (-0.5,-2) rectangle (0.5,-1);
\node (a) at (0,-1.5) {$A$};
\node (b) at (0,-2.5) {$B$};
\draw[->] (0.25,-0.75) -- (1.75,0.75);
\node (alpha) at (0.6,-0.05) {$f$};
\draw[->] (1.25,-1.25) -- (3,0.75);
\node (alpha) at (2.85,0) {$h'\alpha$};
\draw[->] (-0.25,-0.75) -- (-1.75,0.75);
\node (f) at (-0.6,0) {$h$};
\draw[->] (-1.25,-1.25) -- (-3,0.75);
\node (h1alpha) at (-2.85,0) {$h'$};
\end{tikzpicture}
\caption{Diagram of maps in the proof of \autoref{xyunique}. The map $h'\alpha$ is a map of type Y extending the map $f$ of type X, proving that $\mc{N}$ has the XYEP.}\label{xyequivdiag}
\end{figure}
\end{center}

(2) Let $A\subseteq\mc{M},B\subseteq\mc{N}$ and suppose that $f:A\rarr{}B$ is an 
embedding; trivially, $f$ is also a map of type X. We extend $f$ to a map 
$\alpha:\mc{M}\rarr{}\mc{N}$ of type Y via an inductive argument. As $\mc{M}$ is 
countable, we enumerate its elements $M = \{m_0,m_1,\ldots\}$. Set $A = A_0, B = 
B_0$ and $f = f_0$, and suppose that $f_k:A_k\rarr{}B_k$ is a map of type Y 
where $A_i\subseteq A_{i+1}$ and $B_i\subseteq B_{i+1}$ for all $0\leq i\leq 
k-1$. As XY $\in\mf{F}\cap\mf{I}$, $f_k$ is also a map of type X. Select a point 
$m_i\in\mc{M}\smallsetminus A_k$, where $i$ is the least natural number such 
that $m_i\notin A_k$. We see that $A_k\cup\{m_i\}$ is a substructure of $\mc{M}$ 
and is therefore an element of Age$(\mc{N})$ by assumption. As $\mc{N}$ is 
XY-homogeneous, by \autoref{xyep} (1) it has the XYEP. Using this, extend 
$f_k:A_k\rarr{}\mc{N}$ to a map $f_{k+1}:A_k\cup\{m_i\}\rarr{}\mc{N}$ of type Y. 
As XY $\in\mf{F}\cap\mf{I}$, we can repeat this process infinitely many times; 
by ensuring that every $m_i\in\mc{M}$ is included at some stage, we can extend 
the map $f$ to a map $\alpha:\mc{M}\rarr{}\mc{N}$ of type Y as required. We can 
use a similar argument to construct a map $\beta:\mc{N}\to\mc{M}$ of type Y; 
therefore $\mc{M}$ and $\mc{N}$ are Y-equivalent.
\end{proof}

\section{Multifunctions and antihomomorphisms}\label{smultanti}

As mentioned in the introduction, homomorphisms are not ``invertible" in 
general. For instance, there could be a homomorphism between two relational 
$\sigma$-structures $\alpha:\mc{A}\to\mc{B}$ sending a non-relation of $\mc{A}$ 
to a relation in $\mc{B}$; that is, such that $\un{a}\notin\ri{\mc{A}}$ but 
$\un{a}\alpha\in\ri{\mc{B}}$. Furthermore, there is no guarantee that the 
homomorphism is even injective; so $\alpha$ could send two points in $\mc{A}$ to 
the same point in $\mc{B}$. In establishing a suitable `back' amalgamation 
property for our Fra\"{i}ss\'{e}-style theorem, both of these considerations 
must be taken into account. This is achieved by the use of the converse of a function.

For a relation $\rho\subseteq X\times Y$, define the \emph{converse} of $\rho$ to be the set \[\rho^c=\{(y,x)\; : \;(x,y)\in \rho\}\subseteq Y\times X\] We say that a relation $\bar{f}\subseteq Y\times X$ is a \emph{partial multifunction} if $(y,x),(z,x)\in \bar{f}$ implies that $y = z$; and that $\bar{f}$ is a \emph{multifunction} if, in addition, for all $y\in Y$ there exists $x\in X$ such that $(y,x)\in \bar{f}$. It is easy to see that $\bar{f}$ is a partial multifunction if and only if it is the converse $f^c$ of a partial function $f$, and that $\bar{f}$ is a multifunction if and only if the partial function $f$ is surjective. A multifunction $\bar{f}\subseteq Y\times X$ is \emph{surjective} if for all $x\in X$ there exists $y\in Y$ such that $(y,x)\in \bar{f}$. Consequently, $\bar{f}$ is a surjective multifunction if and only if it is the converse of a surjective function $f$. It is clear that a (partial) multifunction $\bar{f}$ is a (partial) function if and only if it is the converse $f^c$ of a (partial) injective function $f$. 
We adopt this barred notation throughout the rest of this article; if $f\subseteq X\times Y$ is a function, denote the partial multifunction given by the converse $f^c$ of $f$ by $\bar{f}\subseteq Y\times X$, and vice versa. Note that $\bar{\bar{f}} = f$ for any function $f$.

\begin{example}\label{multifuneg} Let $X = \{1,\,2,\,3,\,4\}$ and $Y = \{a,\,b,\,c,\,d,\,e\}$ be two sets, and suppose that $f = \{(1,b),\,(2,b),\,(3,a),\,(4,c)\}$ is a function. Then the converse $\bar{f}$ of $f$ is a partial multifunction given by $\bar{f} = \{(b,1),\,(b,2),\,(a,3),\,(c,4)\}$ (see \autoref{pmultifun}). By restricting the codomain $Y$ of $f$ to its image $\im{f}$, the resulting function $g:X\to \im{f}$ that behaves like $f$ is surjective. In this case, the converse $\bar{g}:\im{f}\to X$ of $g$ is a surjective and totally defined multifunction (see shaded portion of \autoref{pmultifun}). This technique will be used frequently in \autoref{sbackforth}.

\begin{center}
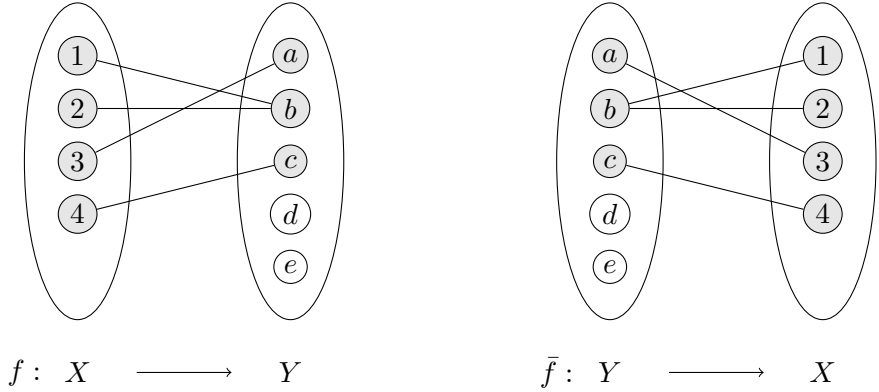
\begin{figure}[h]
\centering
\begin{tikzpicture}[node distance=1.8cm,inner sep=0.7mm,scale=0.7]
\node(U5) at (-3,-1) [circle,draw] {$e$};
\node(U4) at (-3,0) [circle,draw] {$d$};
\node(U3) at (-3,1) [circle,draw,fill=black!10!white] {$c$};
\node(U2) at (-3,2) [circle,draw,fill=black!10!white] {$b$};
\node(U1) at (-3,3) [circle,draw,fill=black!10!white] {$a$};
\node(U) at (-7,-3) {$X$};
\node(f) at (-8,-3) {$f:$};
\node(f) at (-6,-3) {};
\node(to) at (-4,-3) {};
\draw(U3) ellipse (1cm and 3cm);
\node(V1) at (-7,0) [circle,draw,fill=black!10!white] {$4$};
\node(V2) at (-7,1) [circle,draw,fill=black!10!white] {$3$};
\node(V3) at (-7,2) [circle,draw,fill=black!10!white] {$2$};
\node(V4) at (-7,3) [circle,draw,fill=black!10!white] {$1$};
\node(V) at (-3,-3) {$Y$};
\draw(V2) ellipse (1cm and 3cm);
\draw (V4) -- (U2);
\draw (V3) -- (U2);
\draw (V2) -- (U1);
\draw (V1) -- (U3);
\draw[->] (f) -- (to);
\node(U5x) at (3,3) [circle,draw,fill=black!10!white] {$a$};
\node(U4x) at (3,2) [circle,draw,fill=black!10!white] {$b$};
\node(U3x) at (3,1) [circle,draw,fill=black!10!white] {$c$};
\node(U2x) at (3,0) [circle,draw] {$d$};
\node(U1x) at (3,-1) [circle,draw] {$e$};
\node(Ux) at (3,-3) {$Y$};
\node(fx) at (2,-3) {$\bar{f}:$};
\node(fx) at (4,-3) {};
\node(tox) at (6,-3) {};
\draw(U3x) ellipse (1cm and 3cm);
\node(V1x) at (7,3) [circle,draw,fill=black!10!white] {$1$};
\node(V2x) at (7,2) [circle,draw,fill=black!10!white] {$2$};
\node(V3x) at (7,1) [circle,draw,fill=black!10!white] {$3$};
\node(V4x) at (7,0) [circle,draw,fill=black!10!white] {$4$};
\node(Vx) at (7,-3) {$X$};
\draw(V3x) ellipse (1cm and 3cm);
\draw (U4x) -- (V1x);
\draw (U4x) -- (V2x);
\draw (U5x) -- (V3x);
\draw (U3x) -- (V4x);
\draw[->] (fx) -- (tox);
\end{tikzpicture}
\caption[An example of a partial multifunction $\bar{f}$]{Diagram of a function $f$ and its converse, the partial multifunction $\bar{f}$, from \autoref{multifuneg}. The surjective function $g$ and its converse, the surjective multifunction $\bar{g}$, are shaded.}\label{pmultifun}
\end{figure}
\end{center}

\end{example}

If $\bar{f}\subseteq Y\times X$ is a multifunction, we will abuse notation and write $\bar{f}:Y\rarr{}X$ where the context is clear. If $y\in Y$, define the set $y\bar{f} = \{x\in X$ : $(y,x)\in \bar{f}\}$. In a contrast with a function, notice that $y\bar{f}$ is a set and not a single point; in fact, the set could be infinite (see \autoref{infex}). For a tuple $\un{y} = (y_1,\ldots,y_n)\in Y^n$, define $\un{y}\bar{f}$ to be the following set of tuples \[\un{y}\bar{f} = \{(x_1,\ldots,x_n)\mbox{ : }x_i\in y_i\bar{f}\mbox{ for all }1\leq i\leq n\}.\] For a subset $W$ of $Y$, we write \[W\bar{f} = \{x\in X\mbox{ : }(w,x)\in \bar{f}\mbox{ for some }w\in W\} = \bigcup_{w\in W}w\bar{f}.\] Abusing terminology, we say that $Y\bar{f}$ is the \emph{image} of $\bar{f}$.

\begin{example}\label{infex}
Recall that the sign function $s:\mb{R}\to\{-1,0,1\}$ is the surjective function defined by
\[ s(x) = \begin{cases} 
-1 & \text{if }x < 0\\
0 & \text{if }x = 0\\
1 & \text{if }x > 0\\
\end{cases}\] Let $\bar{s}:\{-1,0,1\}\to\mb{R}$ be the corresponding surjective multifunction. For example, $1\bar{s} = \mb{R}^+$ where $\mb{R}^+$ is the set of positive real numbers, $(0,-1)\bar{s} = \{(0,x)\in\mb{R}^2\; :\; x\in\mb{R}^-\}$ where $\mb{R}^-$ is the set of negative real numbers, and $\{-1,1\}\bar{s} = 1\bar{s}\cup -1\bar{s} = \mb{R}\setminus\{0\}$. 
\end{example}

\begin{remark}
 We note the equivalence between the image of the converse of a function $f$ and the preimage of $f$. The reason for the expression of these sets in terms of converses of functions, and not preimages, is for ease of use and notation.
\end{remark}

For a multifunction $\bar{f}:Y\rarr{}X$ and a subset $W\subseteq Y$, we say that the multifunction $\bar{f}|_W:W\to X$ that acts like $\bar{f}$ on $W$ is the \emph{restriction of $\bar{f}$ to $W$}. If $Y\subseteq B$ and $X\subseteq A$ are sets, and $\bar{f}:Y\to X$ and $\bar{g}:B\to A$ are two multifunctions, then we say that $\bar{g}$ \emph{extends} $\bar{f}$ if $y\bar{f} = y\bar{g}$ for all $y\in Y$.

Throughout, we would like to be able to compose functions with multifunctions and vice versa; we achieve this by composing them as relations.

\begin{lemma} \label{compo}\label{conversecompo} 
\begin{enumerate}[(1)]
 \item Suppose that $\bar{g}:C\rarr{}B$ and $\bar{f}:B\rarr{}A$ are multifunctions. Then $\bar{g}\circ \bar{f}=\bar{g}\bar{f}\subseteq C\times A$ is a multifunction.
 \item Let $f:A\to B$ and $g:B\to C$ be two functions, and suppose that $fg:A\to C$ is their composition. Then the converse map $\overline{fg}:C\to A$ is equal to $\bar{g}\bar{f}:C\to A$, where $\bar{g}$ and $\bar{f}$ are composed as relations.\dne
\end{enumerate}
\end{lemma}

\begin{remark} We previously noted that a function $g$ is also a multifunction if and only if it is injective; so by this lemma, the composition of a multifunction $\bar{f}$ with an injective function $g$ (or vice versa) is again a multifunction. Furthermore, the assumption that the function $g$ is injective in this case is necessary for the composition $\bar{f}\circ g$ to be a multifunction.
\end{remark}

Now, we extend the theory of multifunctions into the setting of relational first-order structures.

\begin{definition}\label{defantihomo} Suppose that $\mc{A},\mc{B}$ are two 
$\sigma$-structures and that $\bar{f}:\mc{B}\rarr{}\mc{A}$ is a multifunction. We 
say that $\bar{f}$ is an \emph{antihomomorphism} if
$\neg R_i^\mc{B}(\un{b})$ in $\mc{B}$ then $\neg R_i^\mc{A}(\un{a})$ 
in $\mc{A}$ for all $\un{a}\in \un{b}\bar{f}$ and $R_i\in \sigma$.
\end{definition}

\begin{remark} This definition is equivalent to saying that 
$\bar{f}:\mc{B}\to\mc{A}$ is an antihomomorphism if for all $R_i\in\sigma$ and for 
all $\un{a}\in \un{b}\bar{f}$, then $\ri{\mc{A}}(\un{a})$ implies that 
$\ri{\mc{B}}(\un{b})$.
\end{remark}

Informally, an antihomomorphism is a multifunction that preserves non-relations. The motivation behind this definition is explained by the following alternate 
characterisation of antihomomorphisms.

\begin{lemma}\label{antihomo} Let $\mc{A},\mc{B}$ be two $\sigma$-structures. 
Then $f^c:\mc{B}\rarr{}\mc{A}$ is a surjective antihomomorphism if and only if 
$f:\mc{A}\rarr{}\mc{B}$ is a surjective homomorphism.
\end{lemma}

\begin{proof} Assume that $f:\mc{A}\rarr{}\mc{B}$ is a surjective homomorphism. 
As $f:A\rarr{}B$ is a surjective function we have that $f^c:B\rarr{}A$ is a 
surjective multifunction. Now suppose that $\neg R_i^\mc{B}(\un{b})$. As $f$ 
must preserve relations, we have $\neg R_i^\mc{A}(\un{a})$ whenever $\un{a}f = 
\un{b}$; this is precisely when $\un{a}\in\un{b}f^c$. Conversely, suppose 
that $f^c:\mc{B}\rarr{}\mc{A}$ is a surjective antihomomorphism; therefore 
$f:A\rarr{}B$ is a surjective function. Suppose also that $R_i^\mc{A}(\un{a})$ 
holds. As $f^c$ is an antihomomorphism, it follows that 
$\un{a}\notin\un{b}f^c$ for every $\un{b}$ such that $\neg 
R_i^\mc{B}(\un{b})$. Since $f$ is a function, it must be that 
$\un{a}\in\un{b}f^c$ for some $\un{b}$ such that $R_i^\mc{B}(\un{b})$; so 
$f$ is a homomorphism.
\end{proof}

\begin{remark}
Following this lemma, if $f:A\to B$ is a surjective homomorphism, we write $f^c = \bar{f}: B\to A$ in order to emphasise that the converse of $f$ is both a multifunction and antimonomorphism. Similarly, if $\bar{f}:B\to A$ is a surjective antihomomorphism, we write $\bar{f}^{c} = f: A\to B$ to emphasise that the converse of $\bar{f}$ is both a function and a homomorphism. The context for when we use this notation should be clear.

If $f:\mc{A}\rarr{}\mc{B}$ is any homomorphism, we can restrict the codomain to 
the image to see that $f:\mc{A}\rarr{}\mc{A}f$ is a surjective homomorphism; and 
hence $\bar{f}:\mc{A}f\rarr{}\mc{A}$ is a surjective 
antihomomorphism by \autoref{antihomo}. This technique will be used regularly in \autoref{sbackforth}.
\end{remark}

This result leads to an immediate corollary; an analogue of 
\autoref{conversecompo} (2) for $\sigma$-structures.

\begin{corollary}\label{conversecompofo} Let $\mc{A,B,C}$ be 
$\sigma$-structures, and $f:\mc{A}\to\mc{B}$ and $g:\mc{B}\to\mc{C}$ are 
surjective homomorphisms. Then $\overline{fg} = \bar{g}\bar{f}$ is a surjective 
antihomomorphism.\dne
\end{corollary}

Note that if $f:\mc{A}\rarr{}\mc{B}$ is a bijective homomorphism, then 
$\bar{f}:\mc{B}\rarr{}\mc{A}$ is an bijective function from $B$ to $A$ that 
preserves non-relations; this is the definition of an \emph{antimonomorphism} 
(see \cite{coleman2018permutation}). Furthermore, if 
$f:\mc{A}\rarr{}\mc{B}$ is a isomorphism, then $\bar{f}:\mc{B}\rarr{}\mc{A}$ is 
exactly $f^{-1}$, the inverse isomorphism of $f$. Having determined that the product of two multifunctions is again a multifunction in
\autoref{compo}, an easy composition lemma for antihomomorphisms follows suit.

\begin{lemma}\label{compofo} Let $\mc{A,B,C}$ be $\sigma$-structures. 
Suppose that $\bar{f}:\mc{A}\rarr{}\mc{B}$ and $\bar{g}:\mc{B}\rarr{}\mc{C}$ are 
antihomomorphisms. Then their composition $\bar{f}\bar{g}:\mc{A}\rarr{}\mc{C}$ is an antihomomorphism.\dne
\end{lemma}

\begin{remark} We note that as every antimonomorphism and isomorphism is also 
an antihomomorphism, the product $\bar{f}\bar{g}$ of any antihomomorphism $\bar{f}$ with any 
antimonomorphism or isomorphism $\bar{g}$ is again an antihomomorphism. This fact 
turns out to be crucial in the statement of a suitable amalgamation property for 
the back part of the back-and-forth argument. Furthermore, the product of two 
antimonomorphisms (or an antihomomorphism and an isomorphism) is again an 
antimonomorphism.
\end{remark}

\section{Proof of \autoref{BackForth}}\label{sbackforth}

We now move on to discussing extension of finite partial maps of a $\sigma$-structure to surjective endomorphisms; this is when 
XZ $\in\mf{B}$ (see page \pageref{fbni}, \autoref{fbni}). Due to the lack of symmetry when working with homomorphisms as 
opposed to isomorphisms, we must provide a backwards condition to achieve the 
back part of the required back-and-forth argument. Similar to the more 
conventional amalgamation properties, this backwards condition is defined on 
finite structures. This will involve using the concept of antihomomorphisms 
outlined in \autoref{defantihomo} in three distinct cases; antihomomorphisms 
($\bart{H}$) as converse homomorphisms (H), antimonomorphisms ($\bart{M}$) as the converse of 
monomorphisms (M), and inverse isomorphisms ($\bart{I}$) of isomorphisms (I). Note that the 
classes I and $\bart{I}$ coincide; we use the barred version when applicable 
throughout for notational consistency. It can be seen that 
$\bart{I}\subseteq\bart{M}\subseteq\bart{H}$. We will write $\bar{f}:B\rarr{}A$ to mean some multifunction of type 
$\bart{X}\in\{\bart{H},\bart{M},\bart{I}\}$ from $B$ to $A$ (see 
\autoref{maptype}). 
\begin{figure}[h]
\renewcommand{\arraystretch}{1.2}
\centering
\begin{tabular}{c c c c}
\hline 
Type & Map & Converse type & Converse map \\ 
\hline 
H & homomorphism & $\bart{H}$ & antihomomorphism \\ 
\hline 
M & monomorphism & $\bart{M}$ & antimonomorphism \\ 
\hline 
I & isomorphism & $\bart{I}$ & isomorphism \\
\hline 
\end{tabular}
$\renewcommand{\arraystretch}{2}\qquad\qquad\begin{array}{c|ccc}
 \circ & \bart{H} & \bart{M} & \bart{I} \\
 \hline
 \bart{H} & \bart{H} & \bart{H} & \bart{H} \\
 \bart{M} & \bart{H} & \bart{M} & \bart{M} \\
 \bart{I} & \bart{H} & \bart{M} & \bart{I} \\
\end{array}$ 
\caption{Types of finite partial maps and multifunctions between structures, and how they are composed.}\label{maptype}
\end{figure}

This notation is used in another manner: if $f:A\to B$ is a 
surjective homomorphism of type X, we write $\bar{f}:B\to A$ to be the 
corresponding surjective antihomomorphism of type $\bart{X}$. This is uniquely 
determined by \autoref{antihomo}; see \autoref{maptype} 
for corresponding pairs. The context of when we use this will usually 
be clear. We also recall \autoref{compofo} and its following remarks; the 
composition of two multifunctions of type $\bart{H},\bart{M},\bart{I}$ is again 
a multifunction of type $\bart{H},\bart{M},\bart{I}$; a composition table is given in \autoref{maptype}.

We note that if Z = E then it is a surjective map of type Y = H; likewise, when 
Z = B we have that Y = M and when Z = A we have that Y = I. This relation is 
codified by the following set of pairs:\begin{equation}\label{srel} \ms{S} = 
\{\mbox{(E, H), (B, M), (A, I)}\}.\end{equation} It follows that any XZ-homogeneous 
structure $\mc{M}$ is also XY-homogeneous, where the two are related by the 
relevant pair (Z,Y) $\in\ms{S}$. Therefore, we need to ensure that any 
XZ-homogeneous structure $\mc{M}$ we construct is also XY-homogeneous for the 
appropriate Y; so results in \autoref{sforth} should be satisfied by $\mc{M}$.

As mentioned previously, new properties are required to take care of extension 
and amalgamation in the backwards direction to ensure the map is surjective. 
This is achieved by writing generalised conditions utilising the concept of antihomomorphisms. Throughout, we let X,Y $\in\{$H, M, I$\}$, 
$\bart{X},\bart{Y}\in\{\bart{H},\bart{M},\bart{I}\}$, and Z $\in\{$E, B, A$\}$. 
To avoid any potential confusion, whenever we refer to a map of type Z being a 
surjective map of type Y, the symbol Z is always related to Y in the manner 
illustrated in $\ms{S}$ (see \autoref{srel}).

Motivated by the desire to take care of the `back' part of a back-and-forth argument that would extend a finite partial map of $\mc{M}$ to a surjective endomorphism of $\mc{M}$, we can state the \emph{$\bart{XY}$-extension property} ($\bart{XY}$EP) along similar lines to the XYEP in \autoref{sforth}.

\begin{quotation}($\bart{XY}$EP) Suppose that $\mc{M}$ is a structure with age 
$\ms{C}$. For all $A\subseteq B\in\ms{C}$ and a multifunction 
$\bar{f}:A\rarr{}\mc{M}$ of type $\bart{X}$ such that $A\bar{f}$ is finite, there exists a multifunction 
$\bar{g}:B\rarr{}\mc{M}$ of type $\bart{Y}$ extending $\bar{f}$.
\end{quotation}

Notice that this property differs slightly from previous extension properties as it requires an extra finiteness condition on the image of the multifunction $\bar{f}$. There could exist multifunctions $\bar{f}:A\to\mc{M}$ of type $\bart{H}$ where the image $A\bar{f}$ is an infinite set (see \autoref{infex}); this would cause problems in the proofs of \autoref{xzep} and \autoref{xzexist}.

We turn our attention to finding necessary and sufficient conditions for 
XZ-homogeneity, to be used throughout the proof of \autoref{BackForth}. As 
stated above, we need to ensure that any XZ-homogeneous structure we construct 
is also XY-homogeneous for the appropriate Y. It follows that such a structure 
must satisfy all the conditions outlined in \autoref{xyep}; in particular, XY 
must be in $\mf{I}$ for part (2). With these restrictions in mind, and a desire 
to obtain the most general result possible, we show that both the XYEP and 
$\bart{XY}$EP are necessary conditions for XZ-homogeneity in general, and that 
it these are also sufficient when the extended map of type Y is also a map of 
type X.

\begin{proposition}\label{xzep} Let $\mc{M}$ be a $\sigma$-structure with age 
$\ms{C}$.
\begin{enumerate}[(1)]
\item Suppose that XZ $\in\mf{B}$. If $\mc{M}$ is XZ-homogeneous, then $\mc{M}$ 
has both the XYEP and the $\bart{XY}$EP.
\item Suppose that XZ $\in\mf{B}\cap\mf{I}$. If $\mc{M}$ has the XYEP and the 
$\bart{XY}$EP, then $\mc{M}$ is XZ-homogeneous.
\end{enumerate}
\end{proposition}

\begin{proof} (1) As $\mc{M}$ is XZ-homogeneous, it is also XY-homogeneous and 
so it has the XYEP by \autoref{xyep} (1). Now, suppose that $A,B\in\ms{C}$ and 
$\bar{f}:A\rarr{}\mc{M}$ is a multifunction of type $\bart{X}$ with $A\bar{f}$ finite. As $\ms{C}$ is 
the age of $\mc{M}$, it follows that $\mc{M}$ contains copies $A'\subseteq B'$ 
of $A$ and $B$ and there are isomorphisms $\theta:B\to B'$ and 
$\theta^{-1}:B'\to B$. Restrict the codomain of $\bar{f}$ to its image to find a 
map $\bar{f}':A\rarr{}A\bar{f}$; as this is a surjective multifunction of type 
$\bart{X}$, we have that $\theta^{-1}\bar{f}'=\bar{h}:A'\rarr{}A\bar{f}$ is also a 
surjective multifunction of type $\bart{X}$. By \autoref{antihomo}, the converse 
$h:A\bar{f}\to A'$ of $\theta|_A^{-1}\bar{f}'$ is a surjective map of type X with finite domain $A\bar{f}$; as 
$\mc{M}$ is XZ-homogeneous, extend $h$ to a map $\beta:\mc{M}\to\mc{M}$ of type 
Z. So $\beta\theta^{-1}:\mc{M}\to B$ is a surjective map of type Y; by 
\autoref{conversecompofo}, define $\bar{g} = \theta\bar{\beta}:B\to\mc{M}$ to be 
the corresponding surjective multifunction of type $\bart{Y}$. We need to show 
it extends $\bar{f}$. As $\beta$ extends $h$, then $\bar{\beta}$ extends $\bar{h}$. 
So for all $a\in A$: \[a\bar{f} = a\theta\theta^{-1}\bar{f} = a\theta \bar{h} = 
a\theta\bar{\beta}\] and hence $\mc{M}$ has the $\bart{XY}$EP.\\

(2) Now suppose that XZ $\in\mf{B}\cap\mf{I}$; so a multifunction of type 
$\bart{Y}$ implies that it is also a multifunction of type $\bart{X}$. Suppose 
also that $\mc{M}$ has the XYEP and the $\bart{XY}$EP, and that $f:A\rarr{}B$ is 
a map of type X between substructures of $\mc{M}$. We use a back-and-forth 
argument to show that $\mc{M}$ is XZ-homogeneous.

Set $A=A_0$, $B=B_0$ and $f_0 = f$, and assume that we have extended $f$ to a 
surjective map $f_k:A_k\rarr{}B_k$ of type Y (and hence of type X, by 
assumption), where each $A_i\subseteq A_{i+1}$ and $B_i\subseteq B_{i+1}$ for all 
$0\leq i\leq k-1$. Note also that $A_k$ and $B_k$ are finite for all $k\in\mb{N}$. Furthermore, as $\mc{M}$ is countable we can enumerate the 
elements of $M = \{m_0,m_1,\ldots\}$. 

If $k$ is even, select a point $m_i\in\mc{M}\smallsetminus A_k$ where $i$ is the 
smallest number such that $m_i\notin A_k$, so $A_k\cup\{m_i\}\subseteq\mc{M}$. 
Using the XYEP, extend $f_k$ to a map $f'_{k+1}:A_k\cup\{m_i\}\rarr{}B_k'$ of 
type Y; by restricting the codomain of $f'_{k+1}$ to its image, it follows that 
$f_{k+1}:A_k\cup\{m_i\}\rarr{}B_k\cup\{m_if'_{k+1}\}$ is a surjective map of 
type Y extending $f_k$.

If $k$ is odd, choose a point $m_i\in\mc{M}\smallsetminus B_k$ where $i$ is 
the smallest number such that $m_i\notin B_k$; so 
$B_k\cup\{m_i\}\subseteq\mc{M}$. Note that as $f_k$ is a surjective map of type 
X, we have that $\bar{f}_k:B_k\rarr{}A_k$ is a surjective multifunction of type 
$\bart{X}$. As $A_k$ is finite, we can use the $\bart{XY}$EP to extend $\bar{f}$ to a multifunction 
$\bar{f}_{k+1}':B_k\cup\{m_i\}\to \mc{M}$ of type $\bart{Y}$. Restricting the 
codomain of $\bar{f}_{k+1}$ to its image gives a surjective multifunction 
$\bar{f}_{k+1}:B_k\cup\{m_i\}\to A_k\cup m_i\bar{f}_{k+1}$ of type $\bart{Y}$, 
where $m_i\bar{f}_{k+1}=\{y\in\mc{M} : (y,m_i)\in\bar{f}_{k+1}\}$ is a non-empty 
set. As $\bar{f}_{k+1}$ is a surjective multifunction of type $\bart{Y}$, we 
have that $f_{k+1}:A_k\cup m_i\bar{f}_{k+1}\to B_k\cup\{m_i\}$ is a surjective 
map of type Y extending $f_k$.  

Since XZ $\in\mf{B}\cap\mf{I}$, a map of type Y is also a map of type X; so we 
can use the XYEP and $\bart{XY}$EP to repeat this process infinitely many times. 
By ensuring that each point of $\mc{M}$ appears at both an odd and even step, we 
extend $f$ to a surjective map $\beta$ of type Y; which is a map of 
type Z and so $\mc{M}$ is XZ-homogeneous.
\end{proof}

\begin{remark} Together, \autoref{xyep} and \autoref{xzep} re-prove \cite[Lemma 
1.1]{lockett2014some}, which states that a countable structure $\mc{M}$ is II 
(MI, HI)-homogeneous if and only if it is IA (MA, HA)-homogeneous. For if a 
structure $\mc{M}$ is HI-homogeneous, then it has the HIEP by \autoref{xyep}; 
this implies that every homomorphism between finite substructures of $\mc{M}$ is 
an isomorphism. Since this happens, it follows that every antihomomorphism 
between finite substructures of $\mc{M}$ is an isomorphism. Finally, as $\mc{M}$ 
has the HIEP it must have the $\bart{HI}$EP as well and so $\mc{M}$ is 
HA-homogeneous by \autoref{xzep}. A similar argument works for the equality 
concerning MI-homogeneous structures. In the II case, the IIEP is the standard 
extension property (EP) from Fra\"{i}ss\'{e}'s theorem, and so any structure $\mc{M}$ 
with the IIEP is homogeneous by the same result.
\end{remark}

We now state our new amalgamation property to accommodate the back 
portion of a back-and-forth argument; this is the \emph{$\bart{XY}$-amalgamation 
property} ($\bart{XY}$AP):

\begin{quotation}($\bart{XY}$AP) Let $\ms{C}$ be a class of finite $\sigma$-structures. 
We say that $\ms{C}$ has the $\bart{XY}$AP if for all $A,B_1,B_2\in\ms{C}$, 
multifunction $\bar{f}_1:A\rarr{}B_1$ of type $\bart{X}$ and embedding 
$f_2:A\rarr{}B_2$, there exists a $D\in\ms{C}$, embedding $g_1:B_1\rarr{}D$ and 
multifunction $\bar{g}_2:B_2\rarr{}D$ of type $\bart{Y}$ such that $\bar{f}_1g_1 
= f_2\bar{g}_2$ (see \autoref{dbxyap}).
\end{quotation}

\begin{figure}[h]
\centering
\begin{tikzpicture}[node distance=2cm]
\node(D) {$\exists D$};
\node(R) [below right of=D] {$B_2$};
\node(L) [below left of=D] {$B_1$};
\node(H) [below right of=L] {$A$};
\path[->,thick,dotted] (R) edge node[above right] {$\bar{g}_2$} (D);
\path[right hook->,thick,dotted] (L) edge node[above left] {$g_1$} (D);
\path[right hook->] (H) edge node[below right] {$f_2$} (R);
\path[->] (H) edge node[below left] {$\bar{f}_1$} (L);
\end{tikzpicture}
\caption{The $\bart{XY}$-amalgamation property ($\bart{XY}$AP).}\label{dbxyap}
\end{figure}

Note that this property represents nine different amalgamation conditions. This 
corresponds to one for each class XZ $\in\mf{B}$, where (Z,Y) $\in\ms{S}$ (see 
\autoref{srel} on page \pageref{srel}) and X and $\bart{X}$ are related as in 
\autoref{maptype}. For examples, the $\bart{II}$AP is the standard amalgamation 
property, and the $\bart{MM}$AP is the BAP of \cite{coleman2018permutation}. 

We can now prove \autoref{BackForth} (1). Before we do, we state a straightforward yet important fact about surjective endomorphisms of an infinite first-order structure $\mc{M}$.

\begin{lemma}\label{finitepreimage} Let $\mc{M}$ be a $\sigma$-structure, with 
$\mc{A}$ a finite substructure of $\mc{M}$. Then for any $\alpha\in$ 
Epi$(\mc{M})$, there exists a finite structure $\mc{B}\subseteq \mc{M}$ such 
that $\mc{B}\alpha = \mc{A}$.\dne
\end{lemma}

\begin{proposition}[\autoref{BackForth} (1)]\label{xzap} Suppose that XZ $\in\mf{B}$. If a structure 
$\mc{M}$ is XZ-homogeneous, then Age$(\mc{M})$ has the XYAP and the 
$\bart{XY}$AP.
\end{proposition}

\begin{proof} As $\mc{M}$ is XZ-homogeneous then it is XY-homogeneous and so has 
the XYAP by \autoref{xyap}. To show that Age$(\mc{M})$ has the $\bart{XY}$AP, 
suppose that $A,B_1,B_2\in$ Age$(\mc{M})$, $\bar{f}_1:A\rarr{}B_1$ is a 
multifunction of type $\bart{X}$ and $f_2:A\rarr{}B_2$ is an embedding. We can 
assume without loss of generality that $A,B_1,B_2$ are actually substructures of 
$\mc{M}$ and that $f_2$ is the inclusion mapping. 

By restricting the codomain of $\bar{f}_1$ to its image, 
$\bar{f}_1:A\rarr{}A\bar{f}_1$ is a surjective multifunction of type $\bart{X}$; 
hence the converse $f_1:A\bar{f}_1\rarr{}A$ of $\bar{f}_1$ is a surjective map 
of type X. Use XZ-homogeneity to extend $f_1$ to a map 
$\beta:\mc{M}\rarr{}\mc{M}$ of type Z; and so a surjective map of type Y. We see 
that $B_1\beta$ is a structure containing $A$, and that 
$\beta|_{B_1}:B_1\rarr{}B_1\beta$ extends $f_1$. Define $D = B_1\beta \cup B_2$. 
As $\beta$ is surjective, there exists a finite substructure $C$ such that 
$C\beta = D$ by \autoref{finitepreimage}. Now, define the map $g_1:B_1\rarr{}C$ 
to be the inclusion map. Since $\beta$ is a surjective map of type Y, 
$\bar{\beta}:\mc{M}\rarr{}\mc{M}$ is a surjective multifunction of type 
$\bart{Y}$ by \autoref{antihomo}. Therefore 
$\bar{\beta}|_{B_2}:B_2\rarr{}B_2\bar{\beta}$ is a surjective multifunction of 
type $\bart{Y}$; furthermore, $B_2\bar{\beta}\subseteq C$ as $B_2\subseteq D$. 
Define $\bar{g}_2:B_2\rarr{}C$ to be the multifunction $\bar{\beta}|_{B_2}$ of 
type $\bart{Y}$. It is easy to check that $\bar{f}_1g_1 = f_2\bar{g}_2$ and so 
Age$(\mc{M})$ has the $\bart{XY}$AP. 
\end{proof}

We now show the existence portion of \autoref{BackForth}. Note that the 
previously described inductive construction of an infinite structure in 
\autoref{xyexist} used even and odd steps to achieve different stages of the 
construction at different times. Because we have two amalgamation properties, as 
well as the JEP to ensure a countable structure exists, we proceed using an 
inductive argument at steps congruent to 0, 1, 2 mod 3 to accommodate different 
stages of the construction.

\begin{proposition}[\autoref{BackForth} (2)]\label{xzexist} Suppose that XZ $\in\mf{B}\cap\mf{I}$. Let 
$\ms{C}$ be a class of finite $\sigma$-structures that is closed under 
substructures and isomorphism, has countably many isomorphism types and has the 
JEP, XYAP and the $\bart{XY}$AP. Then there exists a XZ-homogeneous 
$\sigma$-structure $\mc{M}$ with age $\ms{C}$.
\end{proposition}

\begin{proof}
We build $\mc{M}$ over countably many stages, assuming that $M_k$ has been constructed at some stage $k\in\mb{N}$, with $M_0$ being some $A\in\ms{C}$. 
As the number of isomorphism types in $\ms{C}$ is countable, we can choose a countable set $S$ of pairs $(A,B)$ with $A\subseteq B\in\ms{C}$ such that every pair $A'\subseteq B'\in\ms{C}$ is represented by some pair $(A,B)\in S$.  Let $[\mathbf{m}] = \{n\in\mb{N}\; :\; n\equiv m\mod{3}\}$, where $m = 1,2$. Define two bijections $\beta_m:[\mathbf{m}]\times \mb{N}\to [\mathbf{m}]$ such that $\beta_m(i,j) \geq i$ for $m = 1,2$. 

Assume first that $k\equiv 0\mod{3}$. As $\ms{C}$ has countably many isomorphism types, we can enumerate the isomorphism types of $\ms{C}$ by $\{T_0=M_0,T_3,T_6,\ldots\}$. Use the JEP to find a structure $D$ that contains both $M_k$ and a copy of some $A\cong T_k\in\ms{C}$; define $M_{k+1}$ to be this structure $D$. Now suppose that $k\equiv 1\mod{3}$. Let $L_1 = (A_{kj}, B_{kj}, f_{kj})_{j\in\mb{N}}$ be the list of all triples $(A,B,f)$ such that $(A,B)\in S$ and $f:A\to M_k$ is a map of type X. This list is countable as $S$ is and there are finitely many maps of type X from $A$ into $M_k$. Let $k = \beta_1(i,j)$. Then as $\beta_1(i,j)\geq i$, the map $f_{ij}:A_{ij}\to M_i\subseteq M_k$ exists. Therefore, we can use the XYAP to define $M_{k+1}$ such that $M_k\subseteq M_{k+1}$ and the map $f_{ij}:A_{ij}\to M_k$ of type X extends to some map $g_{ij}:B_{ij}\rarr{}M_{k+1}$ of type Y (see \autoref{xyamal} (1)). This ensures that every possible XY-amalgamation occurs. If $k\equiv 2\mod{3}$, let $L_2 = (P_{kj}, Q_{kj}, \bar{f}_{kj})_{j\in\mb{N}}$ be the list of all triples $(P,Q,\bar{f})$ such that $(P,Q)\in S$ and $\bar{f}:P\to M_k$ is a multifunction of type $\bart{X}$ with finite image $P\bar{f}$; again, this list is countable. Let $k = \beta_2(i,j)$; as $\beta_2(i,j)\geq i$ it follows that the multifunction $\bar{f}_{ij}:P_{ij}\to M_i\subseteq M_k$ of type $\bart{X}$ is well-defined. So we can use the XZAP to define $M_{k+1}$ such that $M_k\subseteq M_{k+1}$ and the multifunction $\bar{f}_{ij}:P_{ij}\to M_k$ extends to some multifunction $\bar{g}_{ij}:Q_{ij}\rarr{}M_{k+1}$ (see \autoref{xyamal} (2)). This construction ensures that every possible XZ-amalgamation occurs. 

\begin{figure}[h]
\centering
\begin{tikzpicture}[node distance=2cm,scale=0.9]
\begin{scope}[xshift=-3cm]
\node(D) at (0,1.5) {$M_{k+1}$};
\node(R) at (1.5,0) {$B_{ij}$};
\node(L) at (-1.5,0) {$M_k$};
\node(H) at (0,-1.5) {$A_{ij}$};
\path[->,thick,dotted] (R) edge node[above right] {$g_{ij}$} (D);
\path[right hook->,thick,dotted] (L) edge node[above left] {$\iota_k$} (D);
\path[right hook->] (H) edge node[below right] {$\iota_{ij}$} (R);
\path[->] (H) edge node[below left] {$f_{ij}$} (L);
\node(A) at (0,-2.5) {(1)};
\end{scope}

\begin{scope}[xshift=3cm]
\node(D) at (0,1.5) {$M_{k+1}$};
\node(R) at (1.5,0) {$Q_{ij}$};
\node(L) at (-1.5,0) {$M_k$};
\node(H) at (0,-1.5) {$P_{ij}$};
\path[->,thick,dotted] (R) edge node[above right] {$\bar{g}_{ij}$} (D);
\path[right hook->,thick,dotted] (L) edge node[above left] {$\iota_k$} (D);
\path[right hook->] (H) edge node[below right] {$\iota_{ij}$} (R);
\path[->] (H) edge node[below left] {$\bar{f}_{ij}$} (L);
\node(B) at (0,-2.5) {(2)};
\end{scope}

\end{tikzpicture}
\caption{Amalgamations performed in the proof of \autoref{xzexist}. The $\iota$'s are inclusion mappings.}\label{xyamal}
\end{figure}
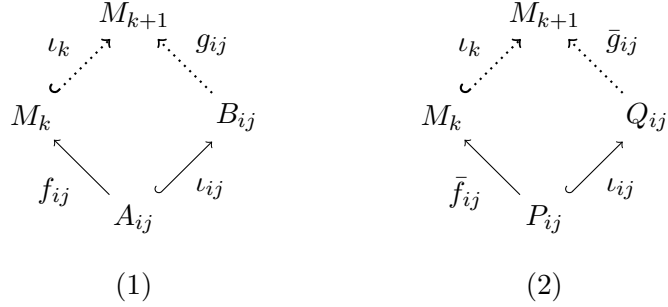

Define $\mc{M} = \bigcup_{k\in\mb{N}}M_k$. Our construction ensures that every isomorphism type of $\ms{C}$ appears at a 0 mod 3 stage, so every structure in $\mc{C}$ embeds into $\mc{M}$. Conversely, we have that $M_k\in\ms{C}$ for all $k\in\mb{N}$. As $\ms{C}$ is closed under substructures, every structure that embeds into $\mc{M}$ is in $\ms{C}$, showing that Age$(\mc{M}) = \ms{C}$. 

It remains to show that $\mc{M}$ is XZ-homogeneous. By \autoref{xzep} and the fact that XZ $\in\mf{B}\cap\mf{I}$, it is enough to show that $\mc{M}$ has both the XYEP and the $\bart{XY}$EP. Assume that $A \subseteq B\in\ms{C}$ and that $f:A\to\mc{M}$ is a map of type X. Using a similar argument to that of \autoref{xyexist} (with $j\in [\mathbf{1}]$ and $(A_{j\ell},B_{j\ell},f_{j\ell})\in L_1$), we can show that as $\ms{C}$ has the XYAP, then $\mc{M}$ has the XYEP.

Now suppose that $P\subseteq Q\in\ms{C}$ and $\bar{f}:P\rarr{}\mc{M}$ is a multifunction of type $\bart{X}$ with finite image $P\bar{f}$. As $P\bar{f}$ is finite, it follows that there exists $u\in [\mathbf{2}]$ such that $P\bar{f}\subseteq M_u$. Furthermore, there exists a triple $(P_{uv},Q_{uv},\bar{f}_{uv})\in L_2$ such that there exists an isomorphism $\eta:Q\to Q_{uv}$ with $P\eta|_P = P_{uv}$ and $\bar{f} = \eta|_P\bar{f}_{uv}$. Define $w = \beta_2(u,v)$; as $w \geq u$, then $P\bar{f} = P_{uv}\bar{f}_{uv}\subseteq M_u \subseteq M_w$. Here, $M_{w+1}$ is constructed by $\bart{XY}$-amalgamating $M_w$ and $Q_{uv}$ over $P_{uv}$; this amalgamation extends the multifunction $\bar{f}_{uv}$ of type $\bart{X}$ to the multifunction $\bar{g}_{uv}:Q_{uv}\to M_{w+1}$ of type $\bart{Y}$. Consequently, the multifunction $\bar{g} = \eta \bar{g}_{uv}:Q\to\mc{M}$ of type $\bart{Y}$ extends the multifunction $\bar{f}$ of type $\bart{X}$ (see \autoref{amaldiag} for a diagram). 
As $\mc{M}$ has both the XYEP and the $\bart{XY}$EP, it follows that $\mc{M}$ is XZ-homogeneous by \autoref{xzep}.
\end{proof}

\begin{center}
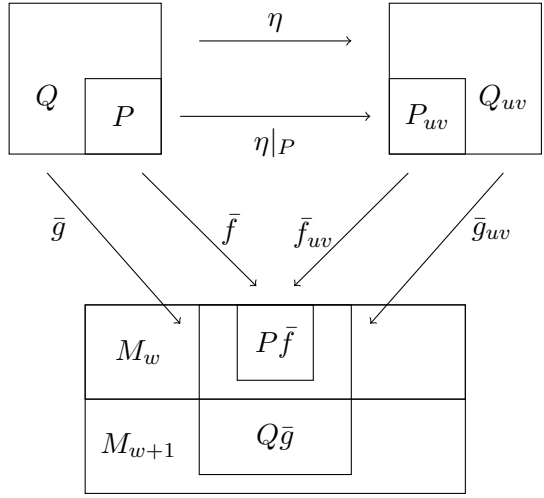
\begin{figure}[h]
\centering
\begin{tikzpicture}[node distance=1.8cm,inner sep=0.7mm,scale=1]

\draw (-3.5,3) rectangle (-1.5,1);
\draw (-2.5,2) rectangle (-1.5,1);
\node (a11) at (-2,1.5) {$P$};
\node (b11) at (-3,1.75) {$Q$};

\draw (3.5,3) rectangle (1.5,1);
\draw (2.5,2) rectangle (1.5,1);
\node (a1) at (2,1.5) {$P_{uv}$};
\node (b) at (3,1.75) {$Q_{uv}$};

\draw[->] (-1.25,1.5) -- (1.25,1.5);
\node (e1) at (0,1.15) {$\eta|_P$};

\draw[->] (-1,2.5) -- (1,2.5);
\node (e) at (0,2.75) {$\eta$};

\draw (-2.5,-1) rectangle (2.5,-2.25);
\draw (-2.5,-1) rectangle (2.5,-3.5);
\node at (-1.8,-1.625) {$M_w$};
\node at (-1.8,-2.875) {$M_{w+1}$};

\draw (-1,-3.25) rectangle (1,-1);
\draw (-0.5,-2) rectangle (0.5,-1);
\node (a) at (0,-1.5) {$P\bar{f}$};
\node (b) at (0,-2.75) {$Q\bar{g}$};
\draw[<-] (0.25,-0.75) -- (1.75,0.75);
\node (alpha) at (0.5,-0.05) {$\bar{f}_{uv}$};
\draw[<-] (1.25,-1.25) -- (3,0.75);
\node (alpha) at (2.85,0) {$\bar{g}_{uv}$};
\draw[<-] (-0.25,-0.75) -- (-1.75,0.75);
\node (f) at (-0.6,0) {$\bar{f}$};
\draw[<-] (-1.25,-1.25) -- (-3,0.75);
\node (h1alpha) at (-2.85,0) {$\bar{g}$};
\draw (-2.5,-2.25) -- (2.5,-2.25);
\end{tikzpicture}
\caption{Diagram of maps in the proof of \autoref{xzexist}. The multifunction $\bar{g} = \eta \bar{g}_{uv}:Q\to\mc{M}$ of type $\bart{Y}$ extends the multifunction $\bar{f}$ of type $\bart{X}$, proving that $\mc{M}$ has the $\bart{XY}$EP.}\label{amaldiag}
\end{figure}
\end{center} 

Finally, we show part (3) of \autoref{BackForth}. Using the fact that 
XZ-homogeneous structures have two extension properties, we can ensure that a 
map between two of them is surjective by using a back-and-forth argument. This 
motivates a new definition, building on that of Y-equivalence.

\begin{definition} Let $\mc{M,N}$ be $\sigma$-structures, and suppose that Z 
$\in\{$E, B, A$\}$ corresponds to the surjective map of type Y $\in\{$H, M, 
I$\}$ via the relation $\ms{S}$. Say that 
$\mc{M}$ and $\mc{N}$ are \emph{Z-equivalent} if Age$(\mc{M})$ = Age$(\mc{N})$ 
and every embedding $f:A\rarr{}\mc{N}$ from a finite substructure $A$ of 
$\mc{M}$ into $\mc{N}$ extends to a surjective map $g:\mc{M}\rarr{}\mc{N}$ of 
type Y, and vice versa.
\end{definition}

For an example, $\mc{M,N}$ are B-equivalent means that they are bi-equivalent in 
the sense of \cite{coleman2018permutation}. Note that if two structures $\mc{M}$ and 
$\mc{N}$ are Z-equivalent, then they are also Y-equivalent where (Z,Y) 
$\in\ms{S}$ (from \autoref{srel}).

\begin{proposition}[\autoref{BackForth} (3)]\label{xzunique} Suppose that XZ $\in\mf{B}\cap\mf{I}$.
\begin{enumerate}[(1)]
\item Assume that $\mc{M,N}$ are Z-equivalent. Then $\mc{M}$ is XZ-homogeneous 
if and only if $\mc{N}$ is. 
\item If $\mc{M,N}$ are XZ-homogeneous and Age$(\mc{M})$ = Age$(\mc{N})$, then 
$\mc{M}$ and $\mc{N}$ are Z-equivalent.
\end{enumerate}
\end{proposition}

\begin{proof}
(1) As $\mc{M,N}$ are Z-equivalent they are also Y-equivalent; as $\mc{M}$ is 
also XY-homogeneous, so is $\mc{N}$ by \autoref{xyunique}. By \autoref{xyep}, it follows that $\mc{N}$ has the 
XYEP. We show now that $\mc{N}$ has the $\bart{XY}$EP. Suppose that $A\subseteq B\in$ Age$(\mc{N})$ and there exists a multifunction 
$\bar{f}:A\rarr{}A'\subseteq\mc{N}$ of type $\bart{X}$ with $A'$ finite. Note that $A$ need not 
be isomorphic to $A'$. As Age$(\mc{M})=$ Age$(\mc{N})$ there exists a copy $A''$ 
of $A'$ in $\mc{M}$; fix an isomorphism $e:A'\to A''$ between the two. 
Therefore, $e$ is a isomorphism from a finite structure of $\mc{N}$ into 
$\mc{M}$; as the two are Z-equivalent, we extend this to a surjective map 
$\alpha:\mc{N}\to\mc{M}$ of type Y. This in turn induces a surjective map 
$\bar{\alpha}:\mc{M}\rarr{}\mc{N}$ by \autoref{antihomo}. Note that 
$\bar{\alpha}$ extends the isomorphism $e^{-1}:A''\to A'$. Now, define $\bar{h} = \bar{f}e: A\to A''$; this is a 
multifunction of type $\bart{X}$ from $A$ into $\mc{M}$ with $A\bar{h} = A''$ finite. Since $\mc{M}$ is 
XZ-homogeneous, it has the $\bart{XY}$EP by \autoref{xzep} and so we extend 
$\bar{h}$ to a multifunction $\bar{h}':B\to \mc{M}$ of type $\bart{Y}$. Here, 
the multifunction $\bar{h}'\bar{\alpha}:B\rarr{}\mc{N}$ is also of type 
$\bart{Y}$; we need to show it extends $\bar{f}$. As $\bar{h}'$ extends $\bar{h} 
= \bar{f}e$, it follows that: \[a\bar{f} = a\bar{f}ee^{-1} = 
a\bar{f}e\bar{\alpha} = a\bar{h}'\bar{\alpha}\]  for all $a\in A$. Therefore 
$\mc{N}$ has the XZEP.

(2) It is enough to show that $\mc{N}$ has the XYEP and the $\bart{XY}$EP by 
\autoref{xzep}. We utilise a back-and-forth argument constructing the surjective 
map over infinitely many stages. Let $f:A\rarr{}B$ be a bijective embedding from 
a finite structure $A\subseteq \mc{M}$ to a finite substructure 
$B\subseteq\mc{N}$. Set $A=A_0$, $B=B_0$ and $f = f_0$ and assume that 
$f_k:A_k\rarr{}B_k$ is a surjective map of type Y (and so of type X by 
assumption) extending $f_k$. Note that as both $\mc{M}$ and $\mc{N}$ are 
countable, then there exists enumerations $\mc{M} = \{m_0,m_1,\ldots\}$ and $\mc{N} 
= \{n_0,n_1,\ldots\}$.

If $k$ is even, select a $m_i\in\mc{M}\smallsetminus A_k$, where $i$ is the 
smallest natural number such that $m_i\notin A_k$. So 
$A_k\cup\{m_i\}\subseteq\mc{M}$, and is also in Age$(\mc{N})$ by assumption. As 
$\mc{N}$ is XZ-homogeneous it has the XYEP by \autoref{xyep} and we use this to 
extend $f_k$ to a map $f_{k+1}':A_k\cup\{m_i\}\rarr{}\mc{N}$ of type Y. 
Restricting the codomain of $f_{k+1}'$ to its image yields a surjective map 
$f_{k+1}:A_k\cup\{m_i\}\rarr{}B_k\cup\{m_if_{k+1}\}$ of type Y. If $k$ is odd, 
select a $n_i\in\mc{N}\smallsetminus B_k$ such that $i$ is the smallest natural 
number such that $n_i\notin B_k$. Hence $B_k\cup\{n_i\}\subseteq \mc{N}$ and 
thus it is an element of Age$(\mc{M})$ by assumption. As $f_k$ is a surjective 
map of type Y, its converse $\bar{f}_k:B_k\rarr{}A_k$ is a surjective 
multifunction of type $\bart{Y}$ by \autoref{antihomo}, and of type $\bart{X}$ 
by assumption. As $\mc{M}$ is XZ-homogeneous it has the $\bart{XY}$EP and so we 
can extend $\bar{f}_k$ to a multifunction 
$\bar{f}_{k+1}':B_k\cup\{n_i\}\rarr{}\mc{M}$ of type $\bart{Y}$. By restricting 
the codomain of $\bar{f}_{k+1}'$ to its image, we obtain a surjective 
multifunction $\bar{f}_{k+1}:B_k\cup\{n_i\}\rarr{}A_k\cup n_i\bar{f}_{k+1}'$ of 
type $\bart{Y}$, where $n_i\bar{f}_{k+1}' = \{(n_i,y)$ : $y\in\mc{M}\}$ is a 
non-empty set. So by \autoref{antihomo}, there exists a surjective map 
$f_{k+1}:A_k\cup n_i\bar{f}_{k+1}'\to B_k\cup\{n_i\}$ of type Y extending $f_k$. 
By our earlier assumption, as a map of type Y is also a map of type X, we can 
repeat this process infinitely many times. By ensuring all points in $\mc{M}$ 
appear at even stages and all points in $\mc{N}$ appear at odd stages, we 
construct a surjective map $\alpha:\mc{M}\rarr{}\mc{N}$ of type Y as required. 
We can use a similar method to show that we can extend any embedding 
$g:A\rarr{}B$, where $A\in\mc{N}$ and $B\in\mc{M}$, to a surjective map of type Y; 
proving that $\mc{M}$ and $\mc{N}$ are Z-equivalent.
\end{proof}

Of course, the open problem that arises from Sections \ref{sforth} and \ref{sbackforth} is:

\begin{question} Can we expand \autoref{Forth} and \autoref{BackForth} to 
include those homomorphism-homogeneity classes in $\mf{B}$?
\end{question}

\section{Maximal homomorphism-homogeneity classes}\label{smhhc}

This section is devoted to determining the extent to which well known examples 
of homogeneous structures are also homomorphism-homogeneous. In some cases, 
verifying that a structure $\mc{M}$ is homogeneous involves using a property of 
$\mc{M}$ to determine that $\mc{M}$ has the EP, and so is homogeneous. Good examples of such properties are the density of 
$(\mb{Q},<)$, and \emph{Alice's restaurant property} characteristic of $R$ (see \autoref{mhhrangra}). In the homomorphism-homogeneity case, this idea was used 
by Cameron and Lockett \cite{cameron2010posets} and Lockett and Truss \cite{lockett2014some} to classify homomorphism-homogeneous posets and determine their position relative to the natural containment order on $\mb{H}$ (see \autoref{hhdiag}). In 
addition to this, Dolinka \cite{dolinka2014bergman} used properties of known 
homogeneous structures to show that they satisfied the \emph{one-point 
homomorphism extension property} (1PHEP), a necessary and sufficient condition 
for HH-homogeneity. Our approach in this section is similar to that of Section 3 
of \cite{dolinka2014bergman}; by defining necessary and sufficient conditions 
for XY and XZ-homogeneity and using properties of structures to show that these are 
satisfied or not satisfied. As in \autoref{sbackforth}, we let X,Y $\in\{$H, M, 
I$\}$, $\bart{X},\bart{Y}\in\{\bart{H},\bart{M},\bart{I}\}$, and Z $\in\{$E, B, 
A$\}$ throughout this section. Furthermore, the pair (Z,Y)$\in\ms{S}$ is related 
as in \autoref{srel} on page \pageref{srel}.

So to begin this section, we define the \emph{one-point XY-extension property}, 
and the \emph{one-point $\bart{XY}$-extension property}:

\begin{quotation}(1PXYEP) We say that a $\sigma$-structure $\mc{M}$ with age 
$\ms{C}$ has the 1PXYEP if for all $A\subseteq B\in\ms{C}$ with 
$|B\smallsetminus A| = 1$ and maps $f:A\rarr{}\mc{M}$ of type X, there exists a 
map $g:B\rarr{}\mc{M}$ of type Y extending $f$.
\end{quotation}

\begin{quotation}(1P$\bart{XY}$EP) Suppose that $\mc{M}$ is a $\sigma$-structure with age $\ms{C}$. Say that $\mc{M}$ has the 1P$\bart{XY}$EP if for all $A\subseteq B\in\ms{C}$ with $|B\smallsetminus A| = 1$, and a multifunction $\bar{f}:A\rarr{}\mc{M}$ of type $\bart{X}$ with $A\bar{f}$ finite, there exists a multifunction 
$\bar{g}:B\rarr{}\mc{M}$ of type $\bart{Y}$ extending $\bar{f}$.
\end{quotation}

For an example, the 1PHHEP is the same thing as the 1PHEP of 
\cite{dolinka2014bergman}. These properties, together with the next proposition, 
provide some of the theoretical basis for the examples that follow.

\begin{proposition}\label{onepmbep} Suppose that XY $\in\mf{I}$. A countable 
$\sigma$-structure $\mc{M}$ has the XYEP / $\bart{XY}$EP if and only if it has 
the 1PXYEP / 1P$\bart{XY}$EP.
\end{proposition}

\begin{proof} The forward direction for both the XYEP and $\bart{XY}$EP cases is 
clear. We now aim to show that if $\mc{M}$ has the 1PXYEP then $\mc{M}$ has the 
XYEP. Assume that $A\subseteq B$ with $|B\smallsetminus A| = n$, and $f:A\to 
\mc{M}$ is a map of type X. We prove the result by induction on the size of this 
complement; the base case (where $n$ = 1) is true by the assumption that 
$\mc{M}$ has the 1PXYEP.

So suppose that for some $k\in\mb{N}$, for any $A\subseteq B\in\ms{C}$ where 
$|B\smallsetminus A| = k$ and any map $f:A\to\mc{M}$ of type X can be extended 
to a map $g:B\to\mc{M}$ of type Y. Take $P\subseteq Q\in\ms{C}$ where 
$|Q\smallsetminus P| = k +1$ and $f':P\to\mc{M}$ to be some map of type X. Now, there 
exists $S\in \ms{C}$ containing $P$ such that $|Q\smallsetminus S| = 1$. By the 
inductive hypothesis, we can extend $f'$ to a map $h:S\to\mc{M}$ of type Y. As 
XY $\in\mf{I}$, it follows that $h$ is also a map of type X. Now, using the 
1PXYEP, extend $h$ to a map $g':Q\to\mc{M}$ of type Y. Since $P\subseteq 
S\subseteq Q$ and $g'$ extends $h$ which extends $f'$, we have that $g'$ extends 
$f'$ and so we are done. Using a similar argument, we can show that if $\mc{M}$ 
has the 1P$\bart{XY}$EP then it has the $\bart{XY}$EP.
\end{proof}

\begin{remark} Let XY~$\in\mf{I}$. Together with \autoref{xyep}, this result 
states that a countable structure $\mc{M}$ has the 1PXYEP if and only if 
$\mc{M}$ is XY-homogeneous. Similarly, by \autoref{xzep} a countable structure 
$\mc{M}$ is XZ-homogeneous if and only if it has the 1PXYEP and the 
1P$\bart{XY}$EP, where (Z,Y) are as in $\ms{S}$ (\autoref{srel} on page \pageref{srel}).
\end{remark}

By considering properties of partial maps and endomorphisms of structures, our 
next result places restrictions on certain types of homomorphism-homogeneity. 
We look at structures known as \emph{cores}; a structure 
$\mc{M}$ is a core if every endomorphism of $\mc{M}$ is an embedding \cite{bodirsky2005core}. Widely studied examples of cores include the 
countable dense linear order without endpoints $(\mb{Q},<)$, the complete graph 
on countably many vertices $K^{\aleph_0}$, the $K^n$-free homogeneous graphs for $n\geq 3$ \cite{mudrinski2010notes} and the Henson digraphs $M_T$ \cite{coleman2018permutation}.
This straightforward result includes a restatement of Lemma 1.1 
of \cite{lockett2014some}.

\begin{lemma}\label{123lem} Let $\mc{M}$ be a countable $\sigma$-structure. 
\begin{enumerate}[(1)]
 \item $\mc{M}$ is MI and MA-homogeneous 
(HI and HA-homogeneous) if and only if $\mc{M}$ is IA-homogeneous and every 
finite partial monomorphism (homomorphism) of $\mc{M}$ is an isomorphism.
\item If $\mc{M}$ is HM or HB-homogeneous, then every finite partial 
homomorphism of $\mc{M}$ is also a monomorphism.
\item Let $\mc{M}$ be a core. If there exists a finite partial monomorphism of 
$\mc{M}$ that is not an isomorphism, then $\mc{M}$ is not MH-homogeneous.
\end{enumerate}
\end{lemma}

\begin{proof} (1) is contained in Lemma 1.1 of \cite{lockett2014some}; notice 
that we cannot extend a map that is not a partial isomorphism of $\mc{M}$ to an 
isomorphism of the entire structure $\mc{M}$. The converse direction is clear. 
To show (2), note that if $h$ is a finite partial homomorphism of $\mc{M}$ that 
is not injective, then we cannot possibly extend this to an injective map and so 
$\mc{M}$ does not have the HMEP. For (3), let $h$ be a finite partial 
monomorphism of a core $\mc{M}$ that is not an isomorphism. As any endomorphism 
of $\mc{M}$ is an embedding, we cannot extend $h$.
\end{proof}

\begin{remark} Note that (1) and (2) also follow from \autoref{Forth} and 
\autoref{BackForth}.
\end{remark}

Following the approach of \cite{lockett2014some} in classifying 
homomorphism-homogeneous posets, the idea of this section is to look at 
properties of graphs and digraphs to determine ``maximal" homomorphism-homogeneity 
classes with respect to the containment order on $\mf{H}$. We formally define 
what we mean by ``maximal".

\begin{definition}\label{defmhh} Let $\mc{M}$ be a first-order structure. A 
homomorphism-homogeneity class XY~$\in\mb{H}$ is \emph{maximal} for $\mc{M}$ if 
$\mc{M}$ is XY-homogeneous and $\mc{M}$ is not PQ-homogenenous, where 
PQ~$\subseteq$~XY in $\mb{H}$. If this happens, we say that XY is a 
\emph{maximal homomorphism-homogeneity class} (shortened to \emph{mhh-class}) for $\mc{M}$.
\end{definition}

\begin{remark} While this definition describes a \emph{minimal} element in the 
poset $\mb{H}$, it is so named because of the strengths of different notions of 
homomorphism-homogeneity. For instance, HA-homogeneity is a stronger condition 
than IA-homogeneity, but HA $\subseteq$ IA in $\mb{H}$. This reflects the 
inverse correspondence between the relative strength of notions of 
homomorphism-homogeneity in $\mf{H}$ and containment of classes in $\mb{H}$ (see the discussion on page \pageref{invcor}).
\end{remark}

For example, if $\mc{M}$ is MB-homogeneous but not MA or HB-homogeneous, then MB 
is a mhh-class for $\mc{M}$. A structure $\mc{M}$ may have more than one 
mhh-class. The set of mhh-classes for $\mc{M}$ 
completely determines the extent of homomorphism-homogeneity satisfied by 
$\mc{M}$; we therefore denote this set by $\mb{H}(\mc{M})$. As an example 
$\mb{H}((\mb{Q},<)) = \{$HA$\}$; this example arose from the classification of 
homomorphism-homogeneous posets in \cite{lockett2014some}.

If $\mc{M}$ is a countable $\sigma$-structure where there exists a finite 
partial monomorphism of $\mc{M}$ that is not an isomorphism, and a finite 
partial homomorphism of $\mc{M}$ that is not an monomorphism, then 
\autoref{123lem} implies that the ``best possible" mhh-classes for 
$\mc{M}$ are IA, MB and HE. As an aside, these classes have important roles to 
play in the theory of generic endomorphisms \cite{lockettgeneric}.

Before we investigate some examples in the context of this article, we note the following direct consequence of \autoref{123lem} (3) with respect to \autoref{defmhh}.

\begin{corollary}\label{mhhcores} Let $\mc{M}$ be a countable homogeneous core. If there exists a finite partial monomorphism of 
$\mc{M}$ that is not an isomorphism, then $\mb{H}(\mc{M}) = \{$IA$\}$.\dne
\end{corollary}

\begin{remark} It has been shown that every $K_n$-free graph (for $n \geq 3$) \cite{mudrinski2010notes}, every Henson digraph $M_T$ (for any set of tournaments on more than $3$ vertices) and the myopic local order $S(3)$ \cite{coleman2018permutation} are cores. By \autoref{mhhcores}, it follows that the mhh-class for each of these structures is IA.
\end{remark}

In the rest of this section, we look at a selection of countable homogeneous graphs and digraphs encountered throughout the literature in order to determine sets of mhh-classes for these structures. By restricting ourselves to classes XY 
$\in\mf{I}$, we can recall \autoref{onepmbep} and the remark that follows it; to 
show that $\mc{M}$ is XY-homogeneous it suffices to show that $\mc{M}$ has the 
1PXYEP, and to show that $\mc{M}$ is XZ-homogeneous it suffices to show that it 
has the 1PXYEP and the 1P$\bart{XY}$EP.

\subsection{Graphs}\label{ssgraphs}

In this article, a \emph{graph} $\G$ is a set of \emph{vertices} $V\G$ together with a set of \emph{edges} $E\G$, where this edge set interprets a irreflexive and symmetric binary relation $E$. For $n\in \mb{N}\cup\{\aleph_0\}$, recall that a \emph{complete graph $K^n$ on $n$ vertices} is a graph on vertex set $VK^n$ (with $|VK^n| = n$) with edges given by $u\sim v\in K^n$ if and only if $u\neq v\in VK^n$. The complement of the graph $K^n$ is called the \emph{null graph on $n$ vertices} or an \emph{independent set on $n$ vertices}, and is denoted by $\bar{K}^n$. Recall that the \emph{complement} of a graph $\G = (V,E)$ is the graph $\bar{\G} = (V,[V]^2\smallsetminus E)$; that is, the graph $\G$ on vertex set $V$ where all the edges in $\bar{\G}$ are non-edges of $\G$ and vice versa. A subset $X$ of $V$ is an \emph{independent set} if $x\nsim x'$ for all $x,x'\in X$. For more on the basics of graph theory, see \cite{diestel2000graph}.

\begin{example}\label{mhhinfnul} It is well-known (see \cite{lachlan1980countable}) that the complete graph on countably many vertices $K^{\aleph_0}$ is homogeneous. 
Suppose that $h:A\rarr{}B$ is a homomorphism between two finite substructures of 
$K^{\aleph_0}$. Then as $h$ preserves edges, it cannot send two distinct 
vertices $x_1,x_2\in VA$ to a single point $v\in VB$; hence $h$ is injective. As 
there are no non-edges to preserve, it must preserve non-edges and so $h$ is an 
embedding. It follows from \autoref{123lem} (1) that $K^{\aleph_0}$ is 
HA-homogeneous and so $\mb{H}(K^{\aleph_0}) = \{$HA$\}$. 

Its complement $\bar{K}^{\aleph_0}$, the infinite null graph, is also homogeneous and as every finite partial monomorphism of $\bar{K}^{\aleph_0}$ 
preserves non-edges, it is MA-homogeneous by \autoref{123lem} (1). We note that 
there exist non-injective finite partial homomorphisms of $\bar{K}^{\aleph_0}$ 
and hence it is not HM or HB-homogeneous by \autoref{123lem} (3). So if 
$h:A\rarr{}B$ is any finite partial homomorphism, we can define a bijective map 
$g:\bar{K}^{\aleph_0}\smallsetminus A\rarr{}\bar{K}^{\aleph_0}\smallsetminus B$ 
and note that the map $\alpha:\bar{K}^{\aleph_0}\rarr{}\bar{K}^{\aleph_0}$ that 
acts like $h$ on $A$ and $g$ everywhere else is an epimorphism of 
$\bar{K}^{\aleph_0}$; so $\bar{K}^{\aleph_0}$ is HE. Hence 
$\mb{H}(\bar{K}^{\aleph_0}) = \{$MA, HE$\}$.
\end{example}

\begin{example}\label{mhhrangra} Let $R$ be the random graph (see 
\cite{cameron1997random}, for instance). Note that there exist finite partial monomorphisms of $R$ 
that are not isomorphisms and finite partial homomorphisms of $R$ that are not 
monomorphisms; hence $R$ is not MI or HM-homogeneous by \autoref{123lem}. It was shown in \cite{coleman2018permutation} that $R$ is MB-homogeneous and in \cite{cameron2006homomorphism} that $R$ is HH-homogeneous; here, we show that $R$ 
is HE-homogeneous. To do this, we rely on the \emph{Alice's restaurant property} characteristic of $R$ (see \cite{cameron1997random}), which says:

\begin{quote} (ARP) For any finite, disjoint subsets $U,V\subseteq V\G$, there exists $x\in V\G$ such that $x\sim u$ for all $u\in U$ and $x\nsim v$ for all $v\in V$. 
\end{quote}

Let $A\subseteq B\in$ Age$(R)$ with $B\smallsetminus A = \{b\}$ and suppose that 
$\bar{f}:A\rarr{}R$ is an antihomomorphism such that $A\bar{f}$ is finite. Using ARP, we can find a vertex $v\in VR$ such 
that $v$ is independent of everything in $A\bar{f}$. Let $\bar{g}:B\rarr{}R$ be the multifunction defined by 
such that $b\bar{g} = v$ and $\bar{g}|_A = \bar{f}$; this is an antihomomorphism as all non-edges from $A$ 
to $b$ are preserved. Therefore, $R$ has the 1P$\bart{HH}$EP and so $R$ is HE-homogeneous by \autoref{onepmbep} and 
\autoref{xzep}. We conclude that $\mb{H}(R) = \{$IA, MB, HE$\}$.
\end{example}

\begin{remark} It was shown in \cite[Theorem 5.3]{lockettgeneric} that $R$ has 
a generic endomorphism. As $R$ is HE-homogeneous, it follows from Theorem 2.1 of 
the same source that this generic endomorphism must be in Epi$(R)$.
\end{remark}

\begin{example}\label{mhhchegra} Let $H$ be the complement of the homogeneous 
$K_n$-free graph for $n\geq 3$. It was previously shown that $H$ is MM but not MB-homogeneous \cite{coleman2018permutation}. A result of \cite{dolinka2014bergman} shows that $H$ has the 1PHHEP, and so is both MM and HH-homogeneous by these two results. 

We now show that $H$ does not have the 1P$\bart{HH}$EP and hence cannot be 
HE-homogeneous. Let $A = \{a\}$ be a single vertex, and let $B$ 
be the independent pair of vertices $\{a,b\}$. Note that $A\subseteq B\in$ 
Age$(H)$. Let $\bar{f}:A\rarr{}H$ be an antihomomorphism sending $A =\{a\}$ to an 
independent set $A\bar{f}$ of $n-1$ vertices in $H$; such a substructure exists by 
definition of $H$. Then as antihomomorphisms preserve non-edges and cannot send two points in a domain to a single point in the codomain, a potential image 
point for $b$ in $H$ must be a vertex $x$ independent of $A\bar{f}$; this cannot happen as $H$ would then induce an independent $n$-set. So $H$ does not have the 
1P$\bart{HH}$EP. Therefore, $H$ is not HE-homogeneous and we see that $\mb{H}(H) = 
\{$IA, MM, HH$\}$.
\end{example}

The next result, detailing the rest of the disconnected, countably infinite homogeneous graphs, extend results of \cite{cameron2006homomorphism} and \cite{rusinov2010homomorphism}.

\begin{proposition}\label{mhhdishomo} Let $\Gamma = \bigsqcup_{i\in I}K^n_i$ be a disjoint union of $|I|$ many complete graphs, each of which have size $n\in\mb{N}\cup\{\aleph_0\}$.
\begin{enumerate}[(1)]
 \item If $|I| = m$ for some $m\in\mb{N}$ and $n = \aleph_0$, then $\mb{H}(\G) = \{$IA, MM, HH$\,\}$.
 \item If $|I| = \aleph_0$ and $n\in\mb{N}$, then $\mb{H}(\G) = \{$IA, HE$\,\}$.
 \item If $|I| = n = \aleph_0$, then $\mb{H}(\G) = \{$IA, MB, HE$\,\}$. 
\end{enumerate}
\end{proposition}

\begin{proof}
\begin{enumerate}[(1)]
 \item It is shown in \cite{coleman2018permutation} that the finite disjoint union of infinite complete graphs is MM but not MB-homogeneous. Furthermore, \cite[Proposition 1.1]{cameron2006homomorphism} asserts that $\G$ is HH-homogeneous. It remains to prove that $\G$ is not HE-homogeneous; here, it is enough to show that $\G$ does not have the 1P$\bart{HH}$EP by \autoref{onepmbep} and \autoref{xzep}. As $\G$ in this case does not embed an independent $n+1$-set, the proof of this is similar to \autoref{mhhchegra}.

 \item In this case, $\G$ is HH-homogeneous by \cite[Proposition 1.1]{cameron2006homomorphism}, but not MM-homogeneous by Proposition 2.5 of the same paper. We show that $\G$ has the 1P$\bart{HH}$EP in this case; proving that $\G$ is HE-homogeneous. Suppose that $A\subseteq B\in$ Age$(\G)$ with $B\smallsetminus A = \{b\}$ and that $\bar{f}:A\to \G$ is an antihomomorphism such that $A\bar{f}$ is finite. As $A\bar{f}$ is finite, it follows that $A\bar{f}\subseteq \bigsqcup_{j\in J}K^n_j$ for some finite set $J\subseteq I$. Select a $v\in K^n_k$, where $k\in I\smallsetminus J$; so $v$ is independent of every element of $A\bar{f}$. Define $\bar{g}:B\to \G$ to act like $\bar{f}$ on $A$ with $b\bar{g} = v$. Then $\bar{g}$ is an antihomomorphism and so $\G$ has the 1P$\bart{HH}$EP.
 
 \item The proof that $\G$ is MB-homogeneous can be found in \cite[Proposition 3.4]{coleman2018permutation}; the proof that it is HE-homogeneous is similar to part (2).
\end{enumerate}
\end{proof}

The only countable homogeneous graphs that have not yet been considered are the complements of the disconnected graphs described in \autoref{mhhdishomo}. The following result deals with these cases. Recall from \cite[Theorem 3, Theorem 5]{rusinov2010homomorphism} that a countable graph is MH-homogeneous if and only if it is HH-homogeneous; if this graph is also connected, then it is HH-homogeneous if and only if it is MM-homogeneous.

\begin{proposition}\label{mhhcompdishomo} Let $\bar{\Gamma} = \overline{\bigsqcup_{i\in I}K^n_i}$ be the complement of a disjoint union of $|I|$ many complete graphs, each of which have size $n\in\mb{N}\cup\{\aleph_0\}$.
\begin{enumerate}[(1)]
 \item If $|I| = m$ for some $m\in\mb{N}$ and $n = \aleph_0$, then $\mb{H}(\bar{\G}) = \{$IA$\}$.
 \item If $|I| = \aleph_0$ and $n\in\mb{N}$, then $\mb{H}(\bar{\G}) = \{$IA, MM, HH$\}$.
 \item If $|I| = n = \aleph_0$, then $\mb{H}(\bar{\G}) = \{$IA, MB, HE$\}$. 
\end{enumerate}
\end{proposition}

\begin{proof}
\begin{enumerate}[(1)]
 \item Notice that $\bar{\G}$ does not contain an infinite complete graph; therefore it is not MM-homogeneous by \cite[Proposition 2.5]{cameron2006homomorphism}. Since $\bar{G}$ is connected, it is not MH-homogeneous by the results of \cite{rusinov2010homomorphism} mentioned above.
 
 \item In this case, any finite set $U\subseteq \bar{\G}$ of vertices has a vertex $x$ such that $x$ is adjacent to every element of $U$. (This is property $\tr$ of \cite{coleman2018permutation}.) Therefore, $\bar{G}$ is MM and HH-homogeneous by \cite[Proposition 2.1]{cameron2006homomorphism}. If $\bar{\G}$ were MB-homogeneous, then $\G$ would be MB-homogeneous by \cite[Proposition 3.1]{coleman2018permutation}; contradicting \autoref{mhhdishomo}. So $\G$ is not MB-homogeneous. It remains to show that $\bar{\G}$ is not HE-homogeneous; the proof of this is similar to \autoref{mhhdishomo} (1).
 
 \item This is MB-homogeneous by the remark following \cite[Proposition 3.4]{coleman2018permutation}. It can be shown that this is HE-homogeneous using a similar argument to part (2) of \autoref{mhhdishomo}.
\end{enumerate}
\end{proof}

\begin{remark}
Notice that part (2) of both Propositions \ref{mhhdishomo} and \ref{mhhcompdishomo} show that the complement of a HE-homogeneous graph is not necessarily HE-homogeneous.
\end{remark}

These results mean that we have determined the mhh-classes for all countable homogeneous graphs in Lachlan and Woodrow's classification \cite{lachlan1980countable}; the results are summarised in \autoref{mhhgraphtable}.

\begin{table}[h]
\renewcommand{\arraystretch}{1.8}
\centering
\begin{tabular}{lll}
mhh-classes & countable IA graph & reference \\
\hline \hline
$\{$HA$\}$ & $K^{\aleph_0}$ & (\ref{mhhinfnul})\\ \hline
$\{$MA, HE$\}$ & $\bar{K}^{\aleph_0}$ & (\ref{mhhinfnul})\\ \hline
$\{$IA, MB, HE$\}$ & $R$, $\bigsqcup_{i\in\mb{N}} K^{\aleph_0}_i$, $\overline{\bigsqcup_{i\in\mb{N}} K^{\aleph_0}_i}$ & (\ref{mhhrangra}, \ref{mhhdishomo} (3), \ref{mhhcompdishomo} (3))\\ \hline
$\{$IA, HE$\}$ & $\bigsqcup_{i\in\mb{N}} K^{n}_i$ & (\ref{mhhdishomo} (2))\\ \hline
$\{$IA, MM, HH$\}$ & $\bar{K}^n$-free ($n\geq 3$), $\bigsqcup_{i=1}^n K^{\aleph_0}_i$, $\overline{\bigsqcup_{i\in\mb{N}} K^{n}_i}$ & (\ref{mhhchegra}, \ref{mhhdishomo} (1), \ref{mhhcompdishomo} (2))\\ \hline
$\{$IA$\}$ & $K^n$-free ($n\geq 3$), $\overline{\bigsqcup_{i=1}^n K^{\aleph_0}_i}$ & (\ref{mhhcores}, \ref{mhhcompdishomo} (1))\\ \hline
\end{tabular}
\caption{Maximal homomorphism-homogeneity classes for countable homogeneous graphs.}\label{mhhgraphtable}
\end{table}

\subsection{Digraphs}

Following the results of \autoref{ssgraphs} and similar work on posets \cite{lockett2014some}, the next natural direction would be to determine mhh-classes for the countable homogeneous digraphs. Unless stated otherwise, in this article a \emph{digraph} $\Delta = (V\Delta,A\Delta)$ is a set of \emph{vertices} $V\Delta$ together with a set $A\Delta\subseteq V\Delta^2$ of ordered pairs, called \emph{arcs} of the digraph. For two vertices $x,y$ of $\Delta$, we write $x\to y$ if $(x,y)\in A\Delta$, and $x\parallel y$ if neither $(x,y)$ nor $(y,x)$ are in $A\Delta$. All the digraphs in this article are \emph{loopless}; so for all $x\in \Delta$, it follows that $(x,x)\notin \Delta$. Additionally, we stipulate that they do not contain $2$-cycles -- there is a \emph{2-cycle} between $x$ and $y$ if and only if $x\to y$ and $y\to x$ -- with the notable exception of \autoref{2cycles}.

We provide a few introductory observations here, and leave the further development of the subject as an open question. Our first example deals with the countable homogeneous tournaments, classified in \cite{lachlan1984countable}.

\begin{example}\label{tournament} Recall that a 
\emph{tournament} is defined to be an oriented, loopless complete graph. By a 
similar argument to the complete graph in \autoref{mhhinfnul}, every finite 
partial homomorphism of a tournament is an embedding. It follows from 
\autoref{123lem} (1) that every countable homogeneous tournament is 
HA-homogeneous. Therefore, the three countable homogeneous tournaments as 
classified by Lachlan \cite{lachlan1984countable}, namely $(\mb{Q},<)$, the 
random tournament $\mb{T}$, and the local order $S(2)$, are 
all HA-homogeneous. So HA is the unique mhh-class for these three 
examples.
\end{example}

\begin{example}\label{ograph} Let $D$ be the \emph{generic 
digraph without $2$-cycles}; for a detailed definition, see \cite{agarwal2016reducts}. This is the unique countable homogeneous structure whose age contains all finite digraphs without $2$-cycles. This structure is also MB-homogeneous \cite[Example 4.4]{coleman2018permutation}. 

However, $D$ is not HH-homogeneous. First of all, as there exist finite partial homomorphisms of $D$ that are not monomorphisms (such as an independent $2$-set being mapped to a single point), $D$ is not HM-homogeneous by \autoref{123lem} (2). To show that $D$ is not HH-homogeneous, it is enough to prove that every endomorphism of $D$ is a monomorphism. Consider $\gamma\in$ End$(D)$, and suppose there exists $v\parallel w\in VD$ such that 
$v\gamma = w\gamma$. As $D$ is universal and homogeneous, there exists an 
oriented graph $A = \{v,w,x\}$ such that $x\rarr{}v$ and $w\rarr{}x$ (see \autoref{nastydig}).
\begin{center}
\begin{figure}[h]
\centering
\begin{tikzpicture}[node distance=1.8cm,inner sep=0.7mm,scale=1]
\node(U5) at (-1,0) [circle,draw] {$w$};
\node(U4) at (0,1) [circle,draw] {$x$};
\node(U3) at (1,0) [circle,draw] {$v$};
\draw[middlearrow={latex}] (U5) -- (U4);
\draw[middlearrow={latex}] (U4) -- (U3);
\end{tikzpicture}
\caption{The digraph $A$ described in \autoref{ograph}.}\label{nastydig}
\end{figure}
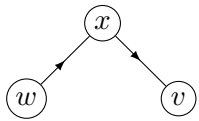
\end{center}
The image of $A$ under $\gamma$ is a 2-cycle and this is a contradiction as $D$ does not 
embed $2$-cycles. It follows that every endomorphism of $D$ is a monomorphism and so $D$ is not HH-homogeneous. We conclude that the mhh-classes of $D$ are IA and MB.
\end{example}

\begin{corollary}\label{ographmm}
Let $\Delta$ be a countable connected homogeneous digraph that embeds the digraph $A$ as in \autoref{nastydig}. Then Mon$(\Delta) =$ End$(\Delta)$ and $\Delta$ is not HH-homogeneous. \dne 
\end{corollary}

\begin{remark} It was shown in \cite{coleman2018permutation} that the $I_n$-free digraph for $n\geq 3$ is MM but not MB-homogeneous. It follows from \autoref{ographmm} that the mhh-classes for these digraphs are IA and MM.
\end{remark}

In the disconnected case, the situation is slightly different. Here is an example of a HE-homogeneous digraph. 

\begin{example} Let $\mb{T}$ be the countable, homogeneous random tournament (see \cite{lachlan1984countable}). Using homogeneity of $\mb{T}$, we can show that $\mb{T}$ satisfies the following property $(\leftrightarrows)$:

\begin{quote}(Property $\lra$)
For all finite, disjoint subsets $U,V$ of $V\mb{T}$ there exists $x\in V\mb{T}$ such that $x\to u$ for all $u$ in $U$ and $v \to x$ for all $v\in V$.
\end{quote}

Now, let $\mc{T} = \bigsqcup_{i\in\mb{N}}\mb{T}_i$ be the infinite disjoint union of isomorphic copies $\mb{T}_i$ of the countable random tournament. Our aim is to show that $\mc{T}$ is HE-homogeneous; so it is enough to show that $\mc{T}$ has both the 1PHHEP and 1P$\bart{HH}$EP by \autoref{onepmbep} and \autoref{xzep}. Suppose then that $A\subseteq B\in$ Age$(\mc{T})$ is such that $B\smallsetminus A = \{b\}$. As $A\in$ Age$(\mc{T})$, we can write $A = \bigsqcup_{j\in J}T_j$, where each $T_j$ is a finite tournament and $J$ is finite. Let $f:A\to\mc{T}$ be a homomorphism. There are two cases to consider; either $b$ is independent of every tournament in $A$, or it is not. If $b$ is independent of every tournament in $A$, then choose any vertex $v\in V\mc{T}\smallsetminus Af$; the function $g:B\to\mc{T}$ acting like $f$ on $A$ and sending $b$ to $v$ is a homomorphism. If $b$ is not independent of $A$, then there is only one tournament $T_j\subseteq A$ that is related to $b$. Partition $T_j$ into two sets \begin{align*} T_j^\rightarrow(b) = \{c\in T_j\stc c\to b\}\quad\hbox{ and }\quad T_j^\leftarrow(b) = \{d\in T_j\stc b\to d\}.
\end{align*}
As $T_j$ is a tournament, it follows that $f|_{T_j}:T_j\to T_jf$ is bijective and so $T_j^\rightarrow(b)f$ and $T_j^\leftarrow(b)f$ partition $T_jf$. As $f:A\to\mc{T}$ is a homomorphism, $T_jf$ is a finite subtournament of $\mb{T}_i$ for some $i\in\mb{N}$; so both $T_j^\rightarrow(b)f$ and $T_j^\leftarrow(b)f$ are finite. We can use property $\lra$ of $\mb{T}_i$ to find a vertex $w$ such that $cf\to w$ for all $cf\in T_j^\rightarrow(b)f$ and $w\to df$ for all $df\in T_j^\leftarrow(b)f$. Now, define a function $g:B\to\mc{T}$ that acts like $f$ on $A$ and sends $b$ to $w$; here, $g$ preserves all relations and so is a homomorphism. Therefore, $\mc{T}$ has the 1PHHEP.

The proof that $\mc{T}$ has the 1P$\bart{HH}$EP follows from a similar argument to \autoref{mhhdishomo} (2). Hence $\mc{T}$ is HE-homogeneous by \autoref{onepmbep} and \autoref{xzep}. Adapting this proof for the 1PMMEP and 1P$\bart{MM}$EP, we can show that $\mc{T}$ is MB-homogeneous; so $\mb{H}(\mc{T}) = \{$IA, MB, HE$\}$.
\end{example}

Changing our definition of digraph to include $2$-cycles also increases the flexibility of the structure.

\begin{example}\label{2cycles} Let $D^*$ be the generic digraph with $2$-cycles; similar to $D$, it is the unique countable homogeneous digraph whose age contains all finite digraphs with $2$-cycles. 
Recall (from \cite[ch4]{mcphee2012endomorphisms} or \cite[ch2]{coleman2017automorphisms}) that $D^*$ has a characteristic extension property known as the \emph{directed 
Alice's restaurant property (DARP)}, which says:

\begin{quote} (DARP) For any finite and pairwise disjoint sets of vertices $U,V,W,X$ of $D^*$, there exists a vertex $z$ of $D^*$ such that: there is an arc from $z$ to every element of $U$, an arc to $z$ from every element of $V$, a 2-cycle between $z$ and every element of $W$, and $z$ is independent of every vertex in $X$. (See \autoref{predarpdi} for a diagram of an example.)
\end{quote}

\begin{center}
\begin{figure}[h]
\centering
\begin{tikzpicture}[node distance=1.8cm,inner sep=0.7mm,scale=1]
\node(R) at (-4.5,-5.5) {$D^*$};
\node(C) at (0,0) {};
\node(U5) at (-1.5,0) [circle,draw] {};
\node(U4) at (-2,0) [circle,draw] {};
\node(U3) at (-2.5,0) [circle,draw] {};
\node(U2) at (-3,0) [circle,draw] {};
\node(U1) at (-3.5,0) [circle,draw] {};
\node(V1) at (1.5,0) [circle,draw] {};
\node(V2) at (2.5,0) [circle,draw] {};
\node(V3) at (3.5,0) [circle,draw] {};
\node(W1) at (-3.5,-4) [circle,draw] {};
\node(W2) at (-2.75,-4) [circle,draw] {};
\node(W3) at (-2,-4) [circle,draw] {};
\node(W4) at (-1.25,-4) [circle,draw] {};
\node(Wx) at (-2.25,-4) {};
\node(X0) at (3.4,-4) [circle,draw] {};
\node(X1) at (3,-4) [circle,draw] {};
\node(X2) at (2.6,-4) [circle,draw] {};
\node(X3) at (2.2,-4) [circle,draw] {};
\node(X4) at (1.8,-4) [circle,draw] {};
\node(X5) at (1.4,-4) [circle,draw] {};
\node(Xx) at (2.4,-4) {};
\node(W) at (-2.25,-5) {$W$};
\node(U) at (-2.5,1) {$U$};
\node(V) at (2.5,1) {$V$};
\node(X) at (2.4,-5) {$X$};
\node(x) at (0,-2) [circle,draw,label=above:$z$] {};
\draw(U3) ellipse (1.8cm and 0.5cm);
\draw(V2) ellipse (1.8cm and 0.5cm);
\draw(Wx) ellipse (1.8cm and 0.5cm);
\draw(Xx) ellipse (1.8cm and 0.5cm);
\draw (-5,-6) rectangle (5,2);
\draw[middlearrow={latex}] (x) -- (U5);
\draw[middlearrow={latex}] (x) -- (U4);
\draw[middlearrow={latex}] (x) -- (U3);
\draw[middlearrow={latex}] (x) -- (U2);
\draw[middlearrow={latex}] (x) -- (U1);
\draw[middlearrow={latex}] (V1) -- (x);
\draw[middlearrow={latex}] (V2) -- (x);
\draw[middlearrow={latex}] (V3) -- (x);
\draw[middlearrow={latex}] (x) to [out=245,in=55] (W4);
\draw[middlearrow={latex}] (x) to [out=234,in=46] (W3);
\draw[middlearrow={latex}] (x) to [out=222,in=38] (W2);
\draw[middlearrow={latex}] (x) to [out=212,in=28] (W1);
\draw[middlearrow={latex}] (W4) to [out=65,in=235] (x);
\draw[middlearrow={latex}] (W3) to [out=54,in=224] (x);
\draw[middlearrow={latex}] (W2) to [out=42,in=213] (x);
\draw[middlearrow={latex}] (W1) to [out=32,in=208] (x);
\draw[dotted] (x) -- (X1);
\draw[dotted] (x) -- (X2);
\draw[dotted] (x) -- (X3);
\draw[dotted] (x) -- (X4);
\draw[dotted] (x) -- (X5);
\draw[dotted] (x) -- (X0);
\end{tikzpicture}
\caption{Example of the directed Alice's restaurant property in the generic digraph with 2-cycles $D^*$.}\label{predarpdi}
\end{figure}
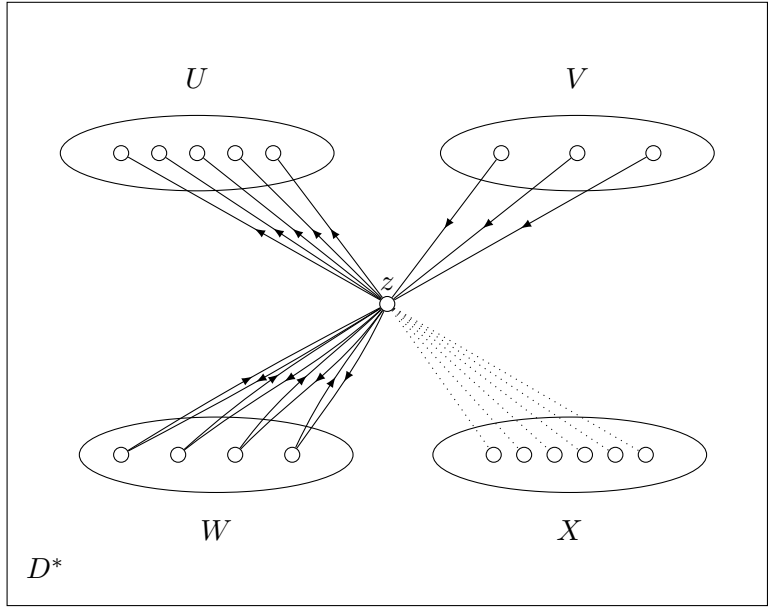
\end{center} 

It was mentioned in \cite{coleman2018permutation} that $D^*$ is MB-homogeneous. Using the DARP, we show that $D^*$ is HE-homogeneous. Let $A\subseteq B\in$ Age$(D^*)$ with $B\smallsetminus A = \{b\}$ and suppose 
that $f:A\rarr{}D^*$ is a homomorphism. As $Af$ is finite, we can use DARP to 
find a vertex $v\in VD^*$ such that there is a 2-cycle between $v$ and every 
element in $Af$. Let $g:B\rarr{}D^*$ be the map such that $bg = v$ 
and $g|_A = f$; this is a homomorphism as all arcs from $A$ to $b$ are 
preserved. Therefore $D^*$ has the 1PHHEP. The proof to show that $D^*$ has the 
1P$\bart{HH}$EP is similar; we use DARP to instead find a vertex $w\in VD^*$
that is independent of the finite set $Af$. The resulting multifunction $\bar{g}$ is 
an antihomomorphism as it preserves all non-relations. Therefore, $D^*$ is 
HE-homogeneous by \autoref{onepmbep} and \autoref{xzep}. So $\mb{H}(D^*) = 
\{$IA, MB, HE$\}$.
\end{example}

\begin{remark} Note the difference between the mhh-classes of $D$, the 
generic digraph without 2-cycles, and $D^*$, the generic digraph with 2-cycles. 
\end{remark}

The work in this section leads to a natural open question.

\begin{question} Investigate countable homomorphism-homogeneous digraphs (both with and without $2$-cycles) in more detail. In particular, determine the mhh-classes for those countable homogeneous digraphs in Cherlin's classification \cite{cherlin1998classification}.
\end{question}

\textbf{Acknowledgements:} The author would like to thank Christian and Maja Pech for pointing out an oversight in the introduction.


\begin{thebibliography}{99}
%
%

\bibitem{agarwal2016reducts}
{L.~Agarwal},
\newblock Reducts of the generic digraph.
\newblock {\em Annals of Pure and Applied Logic}, 167(3):370--391, 2016.

\bibitem{barbina2007reconstruction}
{S. Barbina and H.~D. Macpherson},
\newblock Reconstruction of homogeneous relational structures.
\newblock {\em The Journal of Symbolic Logic}, 72(3):792--802, 2007.

\bibitem{infpermgroups1998}
{M.~Bhattacharjee and H.~D. Macpherson and R.~G. M\"{o}ller and P.~M. Neumann},
\newblock {\em Notes on Infinite Permutation Groups}.
\newblock Number 1698 in Lecture Notes in Mathematics. Springer-Verlag, Berlin, 1998.

\bibitem{bodirsky2005core}
{M.~Bodirsky},
\newblock The core of a countably categorical structure.
\newblock In {\em STACS 2005}, pages 110--120. Springer, 2005.

\bibitem{bodirsky2012complexity}
{M.~Bodirsky},
\newblock Complexity classification in infinite-domain constraint satisfaction.
\newblock {\em arXiv preprint} arXiv:1201.0856, 2012.

\bibitem{oligomorphic1990}
{P.~J. Cameron},
\newblock {\em Oligomorphic permutation groups}, volume 152.
\newblock Cambridge University Press, Cambridge, 1990.

\bibitem{cameron1997random}
{P.~J. Cameron},
\newblock The random graph.
\newblock In {\em The Mathematics of Paul Erd{\"o}s II}, pages 333--351, volume 14, Springer Science and Business Media, 1997.
  
\bibitem{cameron2010posets}
{P.~J. Cameron and D.~C. Lockett},
\newblock Posets, homomorphisms and homogeneity.
\newblock {\em Discrete Mathematics}, 310(3):604--613, 2010.

\bibitem{cameron2006homomorphism}
{P.~J. Cameron and J.~Ne\v{s}et\v{r}il},
\newblock Homomorphism-homogeneous relational structures.
\newblock {\em Combinatorics, Probability \& Computing}, 15(1-2):91--103, 2006.

\bibitem{cherlin1998classification}
{G.~L. Cherlin},
\newblock {\em The classification of countable homogeneous directed graphs and
  countable homogeneous n-tournaments}, volume 621.
\newblock American Mathematical Society, 1998.

\bibitem{coleman2017automorphisms}
 {T.~D.~H. Coleman},
 \newblock {\em Automorphisms and endomorphisms of first-order structures}.
 \newblock PhD thesis, University of East Anglia, 2017.
 
\bibitem{coleman2018permutation}
  {T.~D.~H.~Coleman and D.~M. Evans and R.~D. Gray},
  \newblock {Permutation monoids and MB-homogeneous structures}, 
  \newblock {\em arXiv preprint}, arXiv:1802.04166, 2018.

\bibitem{diestel2000graph}
{R.~Diestel},
\newblock {\em Graph theory}.
\newblock Graduate texts in mathematics, volume 173. Springer-Verlag, Heidelberg, 2000.

\bibitem{dolinka2014bergman}
{I.~Dolinka},
\newblock {The {B}ergman property for endomorphism monoids of some
  {F}ra{\"\i}ss{\'e} limits.}
\newblock {\em Forum Mathematicum}, volume~26, pages 357--376, 2014.

\bibitem{uncountablecofinality2005}
{M. Droste and R. G\"{o}bel},
\newblock {Uncountable cofinalities of permutation groups}
\newblock {\em Journal of the London Mathematical Society},
  71(2):335--344, 2005.

\bibitem{fraisse1953certaines}
{R.~Fra{\"\i}ss{\'e}},
\newblock {Sur certaines relations qui g\'{e}n\'{e}ralisent l'ordre des nombres
  rationnels.}
\newblock {\em Comptes Rendus de l'Ac\'{a}demie des Sciences de Paris},
  237(11):540--542, 1953.

\bibitem{hodges1993model}
{W.~Hodges},
\newblock {\em Model Theory}, volume~42.
\newblock Cambridge University Press, Cambridge, 1993.

\bibitem{hodges1993small}
{W.~Hodges and I. Hodkinson and D. Lascar and S. Shelah},
\newblock {The small index property for $\omega$-stable $\omega$-categorical structures and for the random graph.} 
\newblock {\em Journal of the London Mathematical Society}, 2(2):204--218, 1993.

\bibitem{ilic2008finite}
{A. Ili\'{c} and D. Ma\v{s}ulovi\'{c} and U. Rajkovi\'{c}},
\newblock {Finite homomorphism-homogeneous tournaments with loops}.
\newblock {\em Journal of Graph Theory}, 59(1):45--58, 2008.

\bibitem{kechris2005fraisse}
{A.~Kechris and V.~G. Pestov and S. Todorcevic},
\newblock {Fra\"{i}ss\'{e} limits, {R}amsey theory, and topological dynamics of automorphism groups}.
\newblock {\em Geometric and Functional Analysis}, 15(1):106--189, 2005.

\bibitem{kechris2007turbulence}
{A.~Kechris and C. Rosendal},
\newblock {Turbulence, amalgamation and generic automorphisms of homogeneous structures}.
\newblock {\em Proceedings of the London Mathematical Society}, 94(2):302--350, 2007.

\bibitem{lachlan1984countable}
{A.~H. Lachlan and R.~E. Woodrow},
\newblock {Countable homogeneous tournaments.}
\newblock {\em Transactions of the American Mathematical Society}, 284(2):431--461, 1984.

\bibitem{lachlan1980countable}
{A.~H. Lachlan and R.~E. Woodrow},
\newblock {Countable ultrahomogeneous undirected graphs.}
\newblock {\em Transactions of the American Mathematical Society}, 262(1):51--94, 1980.

\bibitem{lockettgeneric}
{D.~Lockett and J.~K. Truss},
\newblock {Generic endomorphisms of homogeneous structures.}
\newblock {\em Groups and Model Theory, Contemporary Mathematics 576}, pages 217--237, 2012.

\bibitem{lockett2014some}
{D.~C. Lockett and J.~K. Truss},
\newblock {Some more notions of homomorphism-homogeneity}.
\newblock {\em Discrete Mathematics}, 336:69--79, 2014.

\bibitem{macpherson2011survey}
{H.~D. Macpherson},
\newblock {A survey of homogeneous structures.}
\newblock {\em Discrete Mathematics}, 311(15):1599--1634, 2011.

\bibitem{macpherson2011simplicity}
{H.~D. Macpherson and K. Tent},
\newblock {Simplicity of some automorphism groups.}
\newblock {\em Journal of Algebra}, 342(1):40--52, 2011.

\bibitem{semibergman2009}
{V.~Maltcev and J.~D. Mitchell and N. Ru\v{s}kuc},
\newblock {The {B}ergman property for semigroups.} 
\newblock {\em Journal of the London Mathematical Society}, 80(1):212--232, 2009.

\bibitem{masulovic2007homomorphism}
{D.~Ma{\v{s}}ulovi{\'c}},
\newblock {Homomorphism-homogeneous partially ordered sets.}
\newblock {\em Order}, 24(4):215--226, 2007.

\bibitem{masulovic2011oligomorphic}
{D.~Ma{\v{s}}ulovi{\'c} and M. Pech},
\newblock {Oligomorphic transformation monoids and homomorphism-homogeneous structures.}
\newblock {\em Fundamenta Mathematicae}, 212(1):17--34, 2011.

\bibitem{mcphee2012endomorphisms}
{J.~D. McPhee},
\newblock {\em Endomorphisms of Fra{\"\i}ss{\'e} limits and automorphism groups
  of algebraically closed relational structures}.
\newblock PhD thesis, University of St Andrews, 2012.

\bibitem{mudrinski2010notes}
{N.~Mudrinski},
\newblock {Notes on endomorphisms of {H}enson graphs and their complements}
\newblock {\em Ars Combinatoria}, 96:173--183, 2010.

\bibitem{pech2011constraint}
{C. Pech and M. Pech.}
\newblock{Constraint Satisfaction with Weakly Oligomorphic Template}
\newblock {\em preprint}, located at \texttt{http://people.dmi.uns.ac.rs/$\sim$maja/AgeHHnew.pdf}.

\bibitem{pech2016towards}
{C. Pech and M. Pech.} 
\newblock {Towards a Ryll-Nardzewski-type theorem for weakly oligomorphic structures}  
\newblock {\em Math. Log. Q.}, 62(1-2):25–34, 2016.

\bibitem{rusinov2010homomorphism}
{M.~Rusinov and P.~Schweitzer},
\newblock {Homomorphism-homogeneous graphs}.
\newblock {\em Journal of Graph Theory}, 65(3):253--262, 2010.

\bibitem{schmerl1979countable}
{J.~H. Schmerl},
\newblock {Countable homogeneous partially ordered sets}.
\newblock {\em Algebra Universalis}, 9(1):317--321, 1979.

\bibitem{truss1992generic}
{J.~K. Truss},
\newblock {Generic automorphisms of homogeneous structures}.
\newblock {\em Proceedings of the London Mathematical Society}, 3(1):121--141, 1992.
\end{thebibliography}
\end{document}